\documentclass[a4paper,11pt]{amsart}
\usepackage[margin=3.5cm]{geometry}
\usepackage[english]{babel}
\usepackage{amsmath}
\usepackage{amsfonts}
\usepackage{amssymb}
\usepackage{amsthm}
\usepackage{lmodern} \normalfont
\usepackage[T1]{fontenc}
\usepackage{tikz}
\usepackage{color}
\usepackage{hyperref}
\usepackage{enumitem}
\setlist[itemize]{leftmargin=11pt,itemsep=5pt, topsep=5pt}

\usetikzlibrary{positioning,calc}
\usetikzlibrary{shapes}
\DeclareFontShape{T1}{lmr}{bx}{sc} { <-> ssub * cmr/bx/sc }{}
\hypersetup{pdfstartview={XYZ null null 1.00}}

\newtheoremstyle{een}
{11pt}
{11pt}
{\slshape}
{}
{\sc}
{.}
{1mm}
{}

\newtheoremstyle{twee}
{11pt}
{11pt}
{}
{}
{\sc}
{.}
{1mm}
{}

\theoremstyle{een}
\newtheorem{theorem}{\textbf{Theorem}}[section]
\newtheorem{proposition}[theorem]{\textbf{Proposition}}
\newtheorem{lemma}[theorem]{\textbf{Lemma}}

\theoremstyle{twee}
\newtheorem{definition}[theorem]{\textbf{Definition}}
\newtheorem{example}[theorem]{\textbf{Example}}
\newtheorem{remark}[theorem]{\textbf{Remark}}
\newtheorem*{proofoftheorem1}{\textbf{Proof of Theorem \ref{thm1}}}
\newtheorem*{proofoftheorem2}{\textbf{Proof of Theorem \ref{thm2}}}
\newtheorem*{proofofkey2}{\textbf{Proof of Proposition \ref{key2}}}

\makeatletter
\newcommand{\addresseshere}{%
  \enddoc@text\let\enddoc@text\relax
}
\makeatother

\begin{document}
\title{Concurrent lines on del Pezzo surfaces of degree one}

\author{Ronald van Luijk and Rosa Winter}\address{Mathematisch Instituut, Niels Bohrweg 1, 2333 CA Leiden, Netherlands}
\email{rvl@math.leidenuniv.nl}
\address{King's College London, Strand, London, WC2R 2LS, United Kingdom}
\email{rosa.winter@kcl.ac.uk}

\begin{abstract}
Let $X$ be a del Pezzo surface of degree one over an algebraically closed field, and $K_X$ its canonical divisor. The morphism $\varphi$ induced by $|-2K_X|$ realizes $X$ as a double cover of a cone in $\mathbb{P}^3$, ramified over a smooth sextic curve. The surface $X$ contains 240 exceptional curves. We prove the following statements. For a point~$P$ on the ramification curve of $\varphi$, at most sixteen exceptional curves contain~$P$ in characteristic $2$, and at most ten in all other characteristics. Moreover, for a point $Q$ outside the ramification curve, at most twelve exceptional curves contain $Q$ in characteristic $3$, and at most ten in all other characteristics. We show that these upper bounds are sharp, except possibly in characteristic 5 outside the ramification curve. 
\end{abstract}

\maketitle

\section{Introduction}

A del Pezzo surface over a field $k$ is a smooth, projective, geometrically integral surface over $k$ with ample anticanonical divisor. The degree of a del Pezzo surface is the self-intersection number of the canonical divisor, and this is at most 9. Over an algebraically closed field, del Pezzo surfaces of degree $d$ are isomorphic to $\mathbb{P}^2$ blown up at $9-d$ points in general position for $d\neq8$, and to $\mathbb{P}^1\times\mathbb{P}^1$ or $\mathbb{P}^2$ blown up in one point for $d=8$ (\cite{Man74}, Theorem 24.4). For $d\geq3$, a del Pezzo surface of degree~$d$ can be embedded as a surface of degree $d$ in $\mathbb{P}^d$; for example, del Pezzo surfaces of degree three are exactly the smooth cubic surfaces in $\mathbb{P}^3$.\\
An exceptional curve on a del Pezzo surface $X$ with canonical divisor $K_X$ is an irreducible projective curve $C\subset X$ such that $C^2=C\cdot K_X=-1$. For del Pezzo surfaces of degree $d\geq3$, the exceptional curves are exactly the lines on the model of degree $d$ in $\mathbb{P}^d$; for $d=3$ this gives a description of the 27 lines on a cubic surface. The following table shows how many exceptional curves there are for different degrees $d$ over an algebraically closed field.
$$\begin{array}{c|c c c c c c c }
d &  1 & 2 & 3 & 4 & 5 & 6 & 7 \\
\hline
\mbox{exceptional curves} &  240 & 56 & 27 & 16 & 10 & 6 & 3 \\
\end{array}$$

We call a set of exceptional curves \textsl{concurrent} in a point on the surface if that point is contained in all of them. It is well known that on del Pezzo surfaces of degree 3, the number of exceptional curves that are concurrent in a point is at most 3. This can be seen by looking at the graph on the 27 exceptional curves, where two vertices are connected by an edge if the corresponding exceptional curves intersect. For all del Pezzo surfaces of degree 3 this gives the same graph $G$. A set of concurrent exceptional curves corresponds in this way to a complete subgraph of $G$, and the maximal size of complete subgraphs in $G$ is 3. 

On a del Pezzo surface of degree 2, the number of concurrent exceptional curves in a point is at most 4. As in the case for degree 3, this can be derived directly from the intersection graph on the 56 exceptional curves. A geometric argument why 4 is an upper bound is given in \cite{TVAV09}, in the proof of Lemma 4.1. An example where this upper bound is reached is given in \cite{STVA}, Example 2.4. 

\vspace{11pt}

For del Pezzo surfaces of degree 1, the situation is more complex. Contrary to del Pezzo surfaces of degree $\geq2$, for char $k\neq2$, the maximal size of complete subgraphs of the graph on the 240 exceptional curves, which we will show is 16, is not equal to the maximal number of exceptional curves that are concurrent in a point. Let $X$ be a del Pezzo surface of degree 1 over an algebraically closed field $k$, and let $K_X$ be the canonical divisor on $X$. The linear system $|-2K_X|$ gives $X$ the structure of a double cover of a cone $Q$ in $\mathbb{P}^3$, ramified over a smooth curve that is cut out by a cubic surface (\cite{Dem80}, Surfaces de del Pezzo - V, Section 5). Let $\varphi$ be the morphism associated to this linear system. In this article we prove the following two theorems. 

\begin{theorem}\label{thm1} Let $P\in X(k)$ be a point on the ramification curve of $\varphi$. The number of exceptional curves that go through $P$ is at most ten if char $k\neq2$, and at most sixteen if char~$k=2$. 
\end{theorem}

\begin{theorem}\label{thm2} Let $Q\in X(k)$ be a point outside the ramification curve of $\varphi$. The number of exceptional curves that go through $Q$ is at most ten if char $k\neq3$, and at most twelve if char $k=3$.
\end{theorem}

Using the ramification divisor of $\varphi$, we obtain with a simple geometrical argument a slightly weaker upper bound of 12 outside characteristic 2 for Theorem \ref{thm1}, see Remark~\ref{12upperbound}. This was pointed out to us by Niels Lubbes. 

\vspace{11pt}

In \cite{SL14}, Example 4.1, for any field of characteristic unequal to 2, 3, or 5, a del Pezzo surface of degree 1 is defined that contains a point outside the ramification curve that is contained in 10 exceptional curves. This shows that the upper bound for char $k\neq2,3,5$ in Theorem \ref{thm2} is sharp. In Section \ref{examples} we show in all characteristics except for characteristic 5 in the case of Theorem \ref{thm2}, that the upper bounds in Theorems \ref{thm1} and \ref{thm2} are sharp.
Theorems \ref{thm1} and \ref{thm2} are proved by using results on the automorphism group of the graph on the 240 exceptional curves, and by Propositions \ref{key1} and \ref{key2}, which are purely geometrical and show that certain curves in $\mathbb{P}^2$ do not go through the same point.

\vspace{11pt}

The article is organized as follows. We first recall some background and results on the set of exceptional curves on a del Pezzo surface of degree 1 in Section \ref{exccurves}. In this section we mostly use results about del Pezzo surfaces from \cite{Man74}, and results about the weighted graph on the exceptional curves from \cite{WvL}. In Section~\ref{ponthecurve} we prove Theorem \ref{thm1}, and in Section~\ref{poffthecurve} we prove Theorem \ref{thm2}. Finally, Section \ref{examples} contains examples.

\vspace{11pt}

We use \texttt{magma} (\cite{MR1484478}) for our computations, which is the case only in Propositions~\ref{key1} and \ref{key2}. The proofs of Propositions \ref{bovengrens}, \ref{max12bvt}, \ref{trans12}, and \ref{trans11} rely on results in \cite{WvL} that also make use of \texttt{magma}.

\vspace{11pt}

We want to thank Igor Dolgachev and Niels Lubbes for useful discussions and comments. We also want to thank an anonymous referee for giving useful remarks that improved the quality of the paper.
This paper is dedicated to Bas Edixhoven, who we thank posthumously for his insights on a seemingly shorter, but incorrect proof of Theorem \ref{thm1} that had been
presented to us.

\section{A weighted graph on exceptional classes}\label{exccurves}

In this section we recall some results on the Picard group of $X$. We then construct a weighted graph on the 240 exceptional curves, and use the relation between these curves and the $E_8$ root system to study this graph.

\vspace{11pt}

Throughout this paper, when we say that some points in $\mathbb{P}^2$ are in general position, we mean that in the following sense. 

\begin{definition}
Let $r\leq8$ be an integer, and let $P_1,\ldots,P_r$ be points in $\mathbb{P}^2$. 
Then we say that $P_1,\ldots,P_r$ are in \textsl{general position} if there is no line containing three of the points, no conic containing six of the points, and no cubic containing eight of the points with a singularity at one of them. 
\end{definition}

The surface $X$ is isomorphic to $\mathbb{P}^2$ blown up in eight points $P_1,\ldots,P_8$ in general position. Let Pic $X$ be the Picard group of $X$. For $i\in\{1,\ldots,8\}$, let $E_i$ be the class in Pic $X$ corresponding to the exceptional curve above $P_i$, and $L$ the class in Pic $X$ corresponding to the pullback of a line in $\mathbb{P}^2$ not passing through any of the~$P_i$. Then Pic $X$ is isomorphic to $\mathbb{Z}^9$, with basis $\{L,E_1,\ldots,E_8\}$ (\cite{Man74}, Corollary~20.9.1). For $i,j\in\{1,\ldots,8\},\;i\neq j$, we have 
$$E_i^2=-1,\;\;\;\;\;\;\;E_i\cdot E_j=0,\;\;\;\;\;\;\;L\cdot E_i=0,\;\;\;\;\;\;\;L^2=1.$$ Let $K_X$ be the class of a canonical divisor of $X$. Then we have $-K_X=3L-\sum_{i=1}^rE_i$ (\cite{Man74}, Proposition 20.10), hence $-K_X\cdot E_i=1$ for all $i$.

A class $D$ in Pic $X$ with $D^2=-1$ and $D\cdot K_X=-1$ is called an \textsl{exceptional class}, and every exceptional class contains exactly one exceptional curve on $X$ (\cite{Man74}, Theorem 26.2). We know exactly what the exceptional classes in Pic $X$ look like: the following proposition is \cite{Man74}, Proposition 26.1.

\begin{proposition}\label{exceptional}
The exceptional classes in Pic $X$ are the classes of the form  $aL-\sum_{i=1}^8b_iE_i$ where $(a,b_1,\ldots,b_8)$ is given by one of the rows of the following table, 
where all $b_i$ can be permuted.
$$\begin{matrix}
a & b_1 & b_2 & b_3 & b_4 & b_5 & b_6 & b_7 & b_8\\
\hline
0 & -1 & 0 & 0 & 0 & 0 & 0 & 0 & 0 \\
1&1&1&0&0&0&0&0&0\\
2&1&1&1&1&1&0&0&0\\
3&2&1&1&1&1&1&1&0\\
4&2&2&2&1&1&1&1&1\\
5&2&2&2&2&2&2&1&1\\
6&3&2&2&2&2&2&2&2\\
\end{matrix}$$
\end{proposition}

It follows that there are 240 exceptional curves on $X$. 

\begin{remark}\label{configuratie} In 
\cite{Man74}, 26.2, Manin gives a geometrical description of the table in Proposition \ref{exceptional}. An exceptional class of the form $D=aL-\sum_{i=1}^8b_iE_i$, with $(a,b_1,\ldots,b_8)$ a solution given by Proposition \ref{exceptional}, is either one of the $E_i$, where $i\in\{1,\ldots,8\}$ (which is the case if $b_i=-1$), or it corresponds to the class of the strict transform of a curve in $\mathbb{P}^2$ of degree $a$, going through $P_i$ with multiplicity $b_i$ for each~$i$. 
\end{remark}

By $C$ we denote the set of exceptional classes in Pic $X$. Let $U$ be the set $$\{(e_1,e_2,e_3,e_4,e_5,e_6,e_7,e_8)\in C^8\;|\;\forall i\neq j:e_i\cdot e_j=0\}.$$

\begin{lemma}\label{BD}
For $u=(e_1,e_2,e_3,e_4,e_5,e_6,e_7,e_8)\in U$, there exists a morphism $f\colon X\longrightarrow~\mathbb{P}^2$, and points $Q_1,\ldots,Q_8\in \mathbb{P}^2$ that are in general position, such that~$f$ is the blow-up of $\mathbb{P}^2$ at $Q_1,\ldots,Q_8$, and for all $i$, the element $e_i$ corresponds to the class in Pic $X$ of the exceptional curve above $Q_i$. 
\end{lemma}
\begin{proof}An exceptional curve on a surface $Y$ can be blown down in the sense of Castelnuovo (\cite{Har77}, Theorem V.5.7), and if $Y$ is a del Pezzo surface, the resulting surface is a del Pezzo surface too (\cite{Man74}, Corollary 24.5.2 (i)), of degree one higher than $Y$. Since the $e_i$ are disjoint, after blowing down one of them the remaining ones are exceptional curves on the resulting surface, so we can repeatedly blow all eight of them down. It follows that we obtain a morphism $f\colon X\longrightarrow\mathbb{P}^2$, which is the blow-up in eight points $Q_1,\ldots,Q_8$. Since $X$ is a del Pezzo surface it follows that $Q_1,\ldots,Q_8$ are in general position (\cite{Man74}, Theorem 24.3 (ii)). 
\end{proof}

Let $u=(e_1,\ldots,e_8)$ be an element in $U$. By the previous lemma there exists a morphism $f\colon X\longrightarrow\mathbb{P}^2$, and points $Q_1,\ldots,Q_8\in \mathbb{P}^2$ that are in general position, such that $X$ is isomorphic to the blow-up of $\mathbb{P}^2$ at $Q_1,\ldots,Q_8$, and $e_i$ corresponds to the exceptional curve above $Q_i$ for all $i$. It follows that we have $K_X=-3l+\sum_{i=1}^8e_i$, where $l$ corresponds to the strict transform of a line in $\mathbb{P}^2$ not containing any of the~$Q_i$, and $\{l,e_1,\ldots,e_8\}$ forms a basis for Pic $X$.

\begin{remark}\label{bijectionblowdown}
Let $A$ be the set of 240 vectors $(a,b_1,\ldots,b_8)$  that are in the table in Proposition \ref{exceptional} (where the $b_i$ can be permuted). We have a map $$f\colon U\longrightarrow \mbox{Hom}_{\mbox{Set}}(C,A)$$ as follows. Given $u=(e_1,\ldots,e_8)\in U$, let $l$ be the unique element such that $K_X=-3l+\sum_{i=1}^8e_i$. Then we define $f(u)$ as follows. $$f(u)\colon C\longrightarrow A,\;e\longmapsto(e\cdot l,e\cdot e_1,\ldots,e\cdot e_8).$$ The map $f(u)$ is a bijection with inverse $f(u)^{-1}((a,b_1,\ldots,b_8))=al-\sum_{i=1}^8b_ie_i\in C$. Therefore, every element of $U$ gives rise to a bijection between $C$ and $A$.
\end{remark}

Since there is a one-to-one correspondence between $C$ and the set of exceptional curves on $X$, we study the intersection of exceptional curves by studying how elements in $C$ intersect. We do this by constructing a weighed graph on the set $C$.

\begin{definition}
By a \textsl{graph} we mean a pair $(V,D)$, where $V$ is a set of elements called \textsl{vertices}, and $D$ a subset of the powerset of $V$ of which every element has cardinality 2; elements in $D$ are called \textsl{edges}, and the \textsl{size} of the graph is the cardinality of $V$. By a \textsl{weighted graph} we mean a graph $(V,D)$ together with a map $\psi\colon D\longrightarrow A$, where $A$ is any set, whose elements we call \textsl{weights}; for an element $d\in D$ we call $\psi(d)$ its weight. If $(V,D)$ is a weighted graph with weight function~$\psi$, then we define a \textsl{weighted subgraph} of $(V,D)$ to be a graph $(V',D')$ with map $\psi'$, where $V'$ is a subset of $V$, while $D'$ is a subset of the intersection of $D$ with the powerset of $V'$, and $\psi'$ is the restriction of $\psi$ to $D'$. A \textsl{clique} of a weighted graph is a complete weighted subgraph.\end{definition}

By $G$ we denote the complete weighted graph whose vertex set is $C$, and where the weight function is the intersection pairing in Pic $X$. 

\vspace{11pt}

When two exceptional curves intersect in a point on $X$, their corresponding classes in Pic $X$ are connected by an edge of positive weight in $G$. Therefore, an upper bound on the number of exceptional curves on $X$ that are concurrent in a point is given by the maximal size of cliques in $G$ that have only edges of positive weight. To study these cliques, we use the fact there is a one-to-one correspondence between the set~$C$ and the root system $E_8$. We describe this correspondence here.

\vspace{11pt}

Recall that we denote the intersection pairing in Pic $X$ with a dot. Let $\langle\cdot,\cdot\rangle$ be the negative of this intersection pairing on Pic $X$. Then $\langle\cdot,\cdot\rangle$ on $\mathbb{R}\otimes_{\mathbb{Z}}\mbox{Pic }X$ induces the structure of a Euclidean space on the orthogonal complement $K_X^{\perp}$ of the class of the canonical divisor, and with this structure, the set $$E=\{D\in\mbox{Pic }X\;|\;\langle D,D\rangle=2;\;D\cdot K_X=0\}$$ is a root system of type $E_8$ in $K_X^{\perp}$ (\cite{Man74}, Theorem 23.9). The root system $E$ has a system of simple roots given by $E_1-E_2,\;E_2-E_3,\ldots,E_7-E_8,\;L-E_1-E_2-E_3$ (\cite{Man74}, Proposition 25.5.6 (i)). For $c\in C$ we have $c+K_X\in K_X^{\perp}$ and $\langle c+K_X,c+K_X\rangle=2$, and this gives a bijection between $C$ and $E$, sending an exceptional class $c$ to the root $c+K_X$ in~$E$. For $c_1,c_2\in C$ we have $\langle c_1+K_X,c_2+K_X\rangle=1-c_1\cdot c_2$. As a consequence of this bijection, the group of permutations of $C$ that preserve the intersection pairing is isomorphic to the Weyl group $W_8$, which is the group of permutations of $E_8$ generated by the reflections in the hyperplanes orthogonal to the roots (\cite{Man74}, Theorem 23.9). 

\vspace{11pt}

In \cite{WvL}, we studied a complete weighted graph $\Gamma$ which has as vertex set the set of roots in $E$, and weight function the inner product $\langle\cdot,\cdot\rangle$. From the correspondence between $C$ and $E$ it follows that there is bijection between $G$ and $\Gamma$, that sends a vertex $c$ in $G$ to the corresponding vertex $c+K_X$ in $\Gamma$, and an edge $d=\{c_1,c_2\}$ in~$G$ with weight~$w$ to the edge $\delta=\{c_1+K_X,c_2+K_X\}$ in $\Gamma$ with weight $1-w$. The different weights that occur in $G$ are $0,1,2,$ and $3$, and they correspond to weights $1,0,-1,$ and $-2$, respectively, in $\Gamma$. 

\vspace{11pt}

Using the relation between the exceptional classes and the root system $E$, we state some results about $G$. 
 
\begin{lemma}\label{intersecting}
\begin{itemize}
\item[](i) Let $e$ be an exceptional class. Then there is exactly one exceptional class $f$ with $e\cdot f=3$, there are 56 exceptional classes $f$ with $e\cdot f=0$, there are 126 exceptional classes $f$ with $e\cdot f=1$, and 56 exceptional classes $f$ with $e\cdot f=2$. 
\item[](ii) For two exceptional classes $e_1,e_2$ with $e_1\cdot e_2= 2$, there is a unique exceptional class $f$ such that $e_1\cdot f=e_2\cdot f=2$.  
\item[](iii) For every pair $e_1,e_2$ of exceptional classes such that $e_1\cdot e_2=1$, there are exactly~60 exceptional classes $f$ with $e_1\cdot f =e_2\cdot f=1$, and 32 exceptional classes~$f$ with $e_1\cdot f=1$ and $e_2\cdot f=0$. 
\item[](iv) For $e_1,e_2$ two exceptional classes with $e_1\cdot e_2=3$, and $f$ a third exceptional class, we have $e_1\cdot f=1$ if and only if $e_2\cdot f=1$, and $e_1\cdot f=0$ if and only if $e_2\cdot f=2$. 
\end{itemize}
\end{lemma}
\begin{proof}
This is all in \cite{WvL}, using the fact that two exceptional classes have intersection pairing $a$ if and only if their corresponding roots in $E$ have inner product~$1-a$; (i) is Proposition 2.2, (ii) is Lemma 3.8, and (iii) is Lemma 3.26 and Lemmas 3.11 and 3.12. Finally, (iv) follows from the fact that two classes $e_1,e_2$ with $e_1\cdot e_2=3$ correspond to two roots in $E$ with inner product $-2$, which implies they are each other's inverse as vectors (Proposition 2.2 in \cite{WvL}).
\end{proof}

We also obtain a first upper bound for the number of exceptional curves that are concurrent in a point on $X$.

\begin{proposition}\label{bovengrens}
The number of exceptional curves that are concurrent in a point on $X$ is at most~16.
\end{proposition}
\begin{proof}
Cliques with edges of positive weight in $G$ correspond to cliques with edges of weights $-2,-1,0$ in $\Gamma$. The maximal size of such cliques in $\Gamma$ is 16 by \cite{WvL}, Appendix A.
\end{proof}

\begin{definition}
For an exceptional class $e$ in Pic $X$, we call the unique exceptional class $e'$ with $e\cdot e'=3$ its \textsl{partner}.
\end{definition}

The graph below summarizes Lemma \ref{intersecting}. Vertices are exceptional classes, and the number in a subset is its cardinality. The number on an edge between two subsets is the intersection pairing of two classes, one from each subset. For $i,j\in\{1,2,3\}$, the exceptional class $e_i'$ is the partner of the class $e_i$, and for $e_i\cdot e_j=2$, the class $e_{i,j}$ is the unique one that intersects both $e_i$ and $e_j$ with multiplicity 2. 

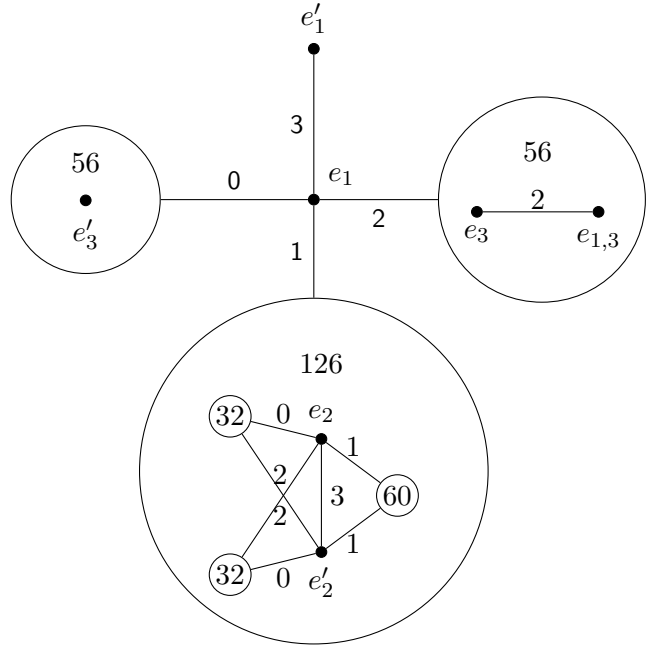
\begin{center}
\begin{tikzpicture}\label{graph}
\node[label=above right:$e_1$,draw,circle,fill,inner sep=0pt,minimum size=4pt](E1) at (5,5){};
\node[label=above:$e_1'$,draw,circle,fill,inner sep=0pt,minimum size=4pt](minE1) at (5,7) {};
\node[draw,circle,inner sep=1pt] (A) at (5,1.4) {
    \begin{tikzpicture} 
	  \node (a1) at (0,0) {126};
      \node [label=above:$e_2$,draw,circle,fill,inner sep=0pt,minimum size=4pt](E2) at (0,-1){};
      \node [label=below:$e_2'$,draw,circle,fill,inner sep=0pt,minimum size=4pt](minE2) at (0,-2.5){};
      \node [draw,circle] (a2) at (1,-1.75) {60};
      \node [draw,circle] (a3) at (-1.2,-0.7) {32};
      \node [draw,circle] (a4) at (-1.2,-2.8) {32};
      \path[every node/.style={}] 
       (E2) edge node [midway,above] {1} (a2)
       (E2) edge node [midway,right] {3} (minE2)
       (E2) edge node [midway,above] {0} (a3)
       (minE2) edge node [midway,below] {0} (a4)
	   (E2) edge node [midway,below] {2} (a4)
       (minE2) edge node [midway,above] {2} (a3)       
       (minE2) edge node [midway,below] {1} (a2);
     \end{tikzpicture}};
\node[draw,circle,inner sep=0pt] (B) at (8,5) {
  \begin{tikzpicture} 
     \node (b1) at (0,0) {56};
     \node [label=below:$e_3$,draw,circle,fill,inner sep=0pt,minimum size=4pt](minE3) at (-0.8,-0.8) {};
     \node [label={[xshift=0cm, yshift=-0.78 cm]$e_{1,3}$},draw,circle,fill,inner sep=0pt,minimum size=4pt](minE1E3) at (0.8,-0.8) {};
     \path[every node/.style={}] 
       (minE3) edge node [midway,above] {2} (minE1E3);
   \end{tikzpicture}};
 \node[draw,circle,inner sep=2pt] (C) at (2,5) {
   \begin{tikzpicture} 
     \node (c1) at (0,0) {56};
     \node [label=below:$e_3'$,draw,circle,fill,inner sep=0pt,minimum size=4pt](E3) at (0,-0.5) {};
   \end{tikzpicture}
    };
   \path[every node/.style={font=\sffamily\small}]
     (E1) edge node [midway,left]{3} (minE1)
     (E1) edge node [midway,below]{2} (B)
     (E1) edge node [midway,left]{1} (A)
     (E1) edge node [midway,above]{0} (C);
\end{tikzpicture}
\end{center}

\vspace{11pt} 

Let $\varphi$ be the morphism associated to the linear system $|-2K_X|$, which realizes $X$ as a double cover of a cone $Q$ in $\mathbb{P}^3$. We want to distinguish cliques in $G$ corresponding to exceptional curves that intersect in a point on the ramification curve of $\varphi$ from those intersecting in a point outside the ramification curve of $\varphi$. To this end we use Proposition \ref{cor}.
 
\begin{proposition}\label{cor} $ $
\begin{itemize}
\item[](i) If $e$ is an exceptional curve on $X$, then $\varphi(e)$ is a smooth conic, the intersection of $Q$ with a plane in $\mathbb{P}^3$ not containing the vertex of $Q$. Moreover $\varphi|_e\colon e\longrightarrow\varphi(e)$ is one-to-one. 
\item[](ii) If $H$ is a hyperplane section of $Q$ not containing the vertex of $Q$, then $\varphi^*H$ has an exceptional curve as component if and only if it has at least three (maybe infinitely near) singular points. If this is the case, then $\varphi^*H=e_1+e_2$ with $e_1,\;e_2$ exceptional curves, and $e_1\cdot e_2=3$. Every exceptional curve arises this way. 
\end{itemize}
\end{proposition}\begin{proof}
\cite{CO99}, Proposition 2.6 and Key-lemma 2.7.
\end{proof}
  
\begin{remark}\label{rem} Let $e$ be an exceptional curve on $X$, and let $e'$ be its partner. Let $H$ be a hyperplane section of $Q$ with $\varphi^*H=e+e'$, which exists by Proposition \ref{cor} (ii). Since~$\varphi|_{f}$ is one-to-one for $f\in\{e,e'\}$ by part (i) of the same proposition, it follows that $\varphi(e)=\varphi(e')=H$. So every point on $H$ has two preimages under $\varphi$, except for the points with a preimage in $e\cap e'$. We conclude that the points where $e$ intersects the ramification curve of $\varphi$ are exactly the points in $e\cap e'$, hence are also contained in $e'$. Conversely, if a set of exceptional curves is concurrent in a point $P$, and this set contains an exceptional curve and its partner, then $P$ lies on the ramification curve of $\varphi$. 
\end{remark}

\section{Proof of Theorem \ref{thm1}}\label{ponthecurve}

In this section we prove Theorem \ref{thm1}. We first determine which cliques in~$G$ may correspond to sets of exceptional curves intersecting on the ramification curve of~$\varphi$ (Remark~\ref{partners}). We then show that the automorphism group of $G$ acts transitively on certain cliques of that form (Proposition~\ref{transitief}), which allows us to reduce to specific curves on $X$. In Proposition~\ref{key1}, which is key to the proof of Theorem \ref{thm1}, we show that seven specific curves are not concurrent. 

\begin{remark}\label{12upperbound}From Remark \ref{rem} it follows that there is a bijection between planes in $\mathbb{P}^3$ that are tritangent to the branch curve of $\varphi$ and do not contain the vertex of~$Q$, and pairs of exceptional curves $e_1,e_2$ with $e_1\cdot e_2=3$. Using this, we can find an upper bound for the number of exceptional curves that are concurrent in a point on the ramification curve. Let $P$ be a point on the branch curve of $\varphi$. From Lemma~4.5 in \cite{TVAV09}, it follows that over a field of characteristic unequal to 2, there are at most 7 planes that are tangent to the branch curve at $P$ and two other points. Moreover, Niels Lubbes gave us the insight that exactly one of those planes contains the vertex of $Q$, so we find an upper bound of 6 planes that are tritangent to the branch curve, that contain $P$, and that do not contain the vertex of $Q$. This gives an upper bound of 12 exceptional curves that contain the point $\varphi^{-1}(P)$ on the ramification curve of $\varphi$, if 
char $k\neq2$. 
\end{remark}

\begin{remark}\label{partners}From Remark \ref{rem} it follows that a maximal set of exceptional curves that are concurrent in a point on the ramification curve consists of exceptional curves and their partners, hence has even size. Moreover, since the weights in such a clique are positive, from Lemma \ref{intersecting} (iv) it follows that such a clique only has edges of weights 1 and 3. We conclude that all cliques in $G$ corresponding to a maximal set of exceptional curves that are concurrent in a point on the ramification curve are of the form $$K_n=\left\{\{e_1,\ldots,e_n,e_1',\dots,e_n'\}\;\left|\begin{array}{c}\forall i:\; e_i,e_i'\in C;\;e_i\mbox{ is the partner of }e_i';\\ \forall i\neq j: e_i\cdot e_j=e_i\cdot e_j'=e_i'\cdot e_j'=1
\end{array}\right.\right\}.$$ \end{remark}

Let $W$ be the group of permutations of $C$ that preserve the intersection pairing, and recall that $W$ is isomorphic to the Weyl group of the $E_8$ root system.

\begin{proposition}\label{transitief}
For $n\in\{2,3,5,6,7,8\}$, the group $W$ acts transitively on the set $K_n$.
\end{proposition}
\begin{proof}
This is Proposition 5.10 in \cite{WvL}.
\end{proof}

We now set up notation for Lemma \ref{genpos1}, which is used in Propositions \ref{key1} and \ref{key2}. Lemma \ref{uniquecubic} is used in Proposition \ref{key1}.

\vspace{11pt}

Let $\mathbb{P}^2$ be the projective plane over $k$ with coordinates $x,y,z$. Let $R_1,\ldots,R_9$ be nine points in $\mathbb{P}^2$, with $R_i=(x_i:y_i:z_i)$ for $i\in\{1,\ldots,9\}$. For $i\in\{1,2,3,4\}$, let Mon$_i$ be the decreasing sequence of $r_i=\binom{i+2}{2}=\tfrac12(i+1)(i+2)$ monomials of degree $i$ in $x,y,z$, ordered lexicographically with $x>y>z$, and for $j\in\{1,\ldots,r_i\}$, let Mon$_i[j]$ be the $j^{\mbox{\tiny{th}}}$ entry of Mon$_i$. For $\delta\in\{x,y,z\},$ let Mon$_i^\delta$ be the list of derivatives of the entries in Mon$_i$ with respect to $\delta$. 
We will define matrices $M,N,L,H$. Note that each row is well defined up to scaling. This means that for all these matrices, the determinant is well defined up to scaling, so asking for the determinant to vanish is well defined. 

\begin{align*}
&M=\left(a_{i,j}\right)_{i,j\in\{1,2,3\}} & &\mbox{with } a_{i,j}=\mbox{Mon}_1[j](R_i);\\
&N=\left(b_{i,j}\right)_{i,j\in\{1,\ldots,6\}} & &\mbox{with } b_{i,j}=\mbox{Mon}_2[j](R_i);\\
&L=\left(c_{i,j}\right)_{i,j\in\{1,\ldots,10\}} & &\mbox{with } c_{i,j} = 
\left\{
	\begin{array}{ll}
		\mbox{Mon}_3[j](R_i)  & \mbox{for } i\leq 8 \\
		\mbox{Mon}_3^{x}[j](R_8)& \mbox{for } i = 9 \\
		\mbox{Mon}_3^{z}[j](R_8)& \mbox{for } i = 10 
	\end{array}
\right..\end{align*}

For $i\in\{7,8,9\}$, let $\alpha_i,\beta_i,\gamma_i$ be such that $\{\alpha_i,\beta_i,\gamma_i\}=\{x,y,z\}$. Define the matrix $$H=\left(d_{i,j}\right)_{i,j\in\{1,\ldots,15\}},$$  
\begin{align*}
& &\mbox{with } d_{i,j}=
\left\{
	\begin{array}{ll}
		\mbox{Mon}_4[j](R_i)  & \mbox{for } i\leq 9 \\
		\mbox{Mon}_4^{\beta_7}[j](R_7)& \mbox{for } i=10\\
		\mbox{Mon}_4^{\gamma_7}[j](R_7)& \mbox{for } i=11\\
		\mbox{Mon}_4^{\beta_8}[j](R_8)& \mbox{for } i=12\\
		\mbox{Mon}_4^{\gamma_8}[j](R_8)& \mbox{for } i=13\\
		\mbox{Mon}_4^{\beta_9}[j](R_9)& \mbox{for } i=14\\
		\mbox{Mon}_4^{\gamma_9}[j](R_9)& \mbox{for } i=15
	\end{array}
\right..\end{align*}

The definitions of $\alpha_i,\beta_i,\gamma_i$ may seem a little confusing; we use them to formulate part (iv) of the following lemma. Note that we choose them before we define the matrix $H$ above, so this matrix is well-defined. 
 
\begin{lemma}\label{genpos1}
The following hold.
\begin{itemize}
\item[](i) The points $R_1,R_2,$ and $R_3$ are collinear if and only if det$(M)=0$.
\item[](ii) The points $R_1,\ldots,R_6$ are on a conic if and only if det$(N)=0$.
\item[](iii) If the points $R_1,\ldots,R_8$ are on a cubic with a singular point at $R_8$, then det$(L)=0$. If $y_8\neq0$, then the converse also holds.
\item[](iv) If the points $R_1,\ldots,R_9$ are on a quartic that is singular at $R_7,\;R_8$ and $R_9$, then det$(H)=0$. If the $\alpha_i$ coordinate of $R_i$ is nonzero for $i\in\{7,8,9\}$, then the converse also holds. 
\end{itemize}
\end{lemma}
\begin{proof}
\begin{itemize}
\item[]
\item[](i) The determinant of $M$ is zero if and only if there is a nonzero element in the nullspace of $M$, that is, there is a nonzero vector $(m_1,m_2,m_3)$ such that for all $i\in \{1,2,3\}$, we have $m_1a_{i,1}+m_2a_{i,2}+m_3a_{i,3}=0$. But this is the case if and only if the line defined by $m_1x+m_2y+m_3z$ contains all three points. 
\item[](ii) This proof goes analogously to the proof of (i).
\item[](iii) The determinant of $L$ is zero if and only if there is a nonzero vector $(l_1,\ldots,l_{10})$ in $k^{10}$ such that for all $i\in \{1,\ldots,10\}$, we have $l_1c_{i,1}+\cdots+l_{10}c_{i,10}=0$. This is the case if and only if the cubic $C$ defined by $\lambda= \sum_{i=1}^{10}l_i\mbox{Mon}_3[i]$ contains all eight points $R_1,\ldots,R_8$, and moreover, the derivatives $\lambda_x,\lambda_z$ of $\lambda$ with respect to $x$ and $z$ vanish in $R_8$. So if $R_1,\ldots,R_8$ are on a cubic with a singular point at $R_8$, the determinant of $L$ vanishes. Conversely, if det$(L)=0$ and $y_8\neq0$, since we have $x\lambda_x+y\lambda_y+z\lambda_z=3\lambda$, this implies that also the derivative $\lambda_y$ of $\lambda$ with respect to $y$ vanishes in $R_8$, hence $C$ is singular in $R_8$. 
\item[](iv) The determinant of $H$ is zero if and only if there exists a nonzero vector given by $(h_1,\ldots,h_{15})$ such that for all $i\in \{1,\ldots,15\}$, we have $h_1d_{i,1}+\cdots+h_{15}d_{i,15}=0$. This is the case if and only if the quartic $K$ defined by $\lambda= \sum_{i=1}^{15}h_i\mbox{Mon}_4[i]$ contains $R_1,\ldots,R_9$, and moreover, for $i\in\{7,8,9\}$, the derivatives $\lambda_{\delta}$ for $\delta\in\{x,y,z\}\setminus\{\alpha_i\}$ vanish in $R_i$. So if $R_1,\ldots,R_9$ are on a quartic that is singular at $R_7,\;R_8$ and $R_9$, the determinant of $H$ vanishes. Conversely, if det$(H)=0$ and $\alpha_i\neq0$ for $i\in\{7,8,9\}$, then, since we have $x\lambda_x+y\lambda_y+z\lambda_z=4\lambda$, this implies that also $\lambda_{\alpha_i}$ vanishes in $R_i$ for $i\in\{7,8,9\}$. So $K$ is singular in $R_7$, $R_8$, and $R_9$.\qedhere
\end{itemize} 
\end{proof}

We recall that $k$ is an algebraically closed field, and $\mathbb{P}^2$ the projective plane over $k$.

\begin{lemma}\label{uniquecubic}If $R_1,\ldots,R_7$ are seven distinct points in $\mathbb{P}^2$ such that $R_1,\ldots,R_6$ are in general position, and the line $L$ containing $R_1$ and $R_7$ contains none of the other points, then there is a unique cubic containing all seven points that is singular in~$R_1$. This cubic does not contain $L$.\end{lemma}
\begin{proof}
The linear system of cubics containing $R_1,\ldots,R_7$ is at least two-dimensional. Requiring that a cubic in this linear system is singular in $R_1$ gives two linear conditions, defining a linear subsystem $\mathcal{C}$ of dimension $\geq0$, so there is at least one cubic containing $R_1,\ldots,R_7$ that is singular at $R_1$. \\
Let $D$ be an element of $\mathcal{C}$; we claim that $D$ does not contain the line $L$ that contains $R_1$ and $R_7$. Indeed, if $D$ were the union of $L$ and a conic $C$, then $R_1$ would be contained in $C$ since it is a singular point of $D$. Since the points $R_2,\ldots, R_6$ are not on $L$ by assumption, they would also be contained in $C$, contradicting the fact that $R_1,\ldots,R_6$ are in general position. So $D$ does not contain $L$. Note that this implies that $D$ is smooth in $R_7$, since if it were singular, then $D$ would intersect $L$ with multiplicity at least 4, hence $D$ would contain $L$.\\
Now assume that there is more than one element in $\mathcal{C}$. Then there are two cubics $D_1$ and $D_2$ that contain $R_1,\ldots,R_7$ with a singularity at $R_1$, and whose defining polynomials are linearly independent. By what we just showed, they are not singular in~$R_7$. For $i=1,2$, let $l_i$ be the tangent line to $D_i$ at $R_7$. 
If the equations defining $l_1$ and $l_2$ are not linearly independent, then there is an element $F$ of $\mathcal{C}$ that is singular in $R_7$, giving a contradiction. We conclude that the equations defining $l_1$ and $l_2$ must be linearly independent. Therefore, there is an element $G$ in $\mathcal{C}$ such that the line $L$ through $R_1$ and $R_7$ is the tangent line to $G$ at $R_7$. But then $L$ intersects~$G$ in four points counted with multiplicity, so it is contained in $G$. This contradicts the fact that $G$ is in $\mathcal{C}$. We conclude that there is a unique cubic through $R_1,\ldots,R_7$ that is singular in $R_1$. This cubic does not contain the line through $R_1$ and $R_7$.
\end{proof}

\begin{proposition}\label{key1}
Assume that the characteristic of $k$ is not 2. Let $Q_1,\ldots,Q_8$ be eight points in~$\mathbb{P}^2$ in general position. For $i\in\{1,2,3,4\}$, let $L_i$ be the line through $Q_{2i}$ and $Q_{2i-1}$, and for $i,j\in\{1,\ldots,8\}$, with $i\neq j$, let $C_{i,j}$ be the unique cubic through $Q_1,\ldots,Q_{i-1},Q_{i+1},\ldots,Q_8$ that is singular in $Q_j$, which exists by Lemma~\ref{uniquecubic}. Assume that the four lines $L_1,\;L_2,\;L_3$ and $L_4$ are concurrent in a point~$P$. Then the three cubics $C_{7,8}$, $C_{8,7}$, and $C_{6,5}$ do not all contain $P$. 
\end{proposition}
\begin{proof}
First note that if $P$ were equal to one of the $Q_i$, then three of the eight $Q_i$ would be on a line, which would contradict the fact that $Q_1,\ldots,Q_8$ are in general position. We conclude that $P$ is not equal to one of the $Q_i$. Moreover, if $P$ were collinear with any two of the three points $Q_1,Q_3,Q_5$, say for example with $Q_1$ and $Q_3$, then, since $P$ is also contained in $L_1$ and $L_2$, it would follow that $L_1$ and $L_2$ are equal, giving a contradiction. So $Q_1,Q_3,Q_5$ and $P$ are in general position.\\
Let $(x:y:z)$ be the coordinates in $\mathbb{P}^2$. After applying an automorphism of $\mathbb{P}^2$ if necessary, we can assume that we have 
\begin{align*}
&Q_1=(0:1:1);	& 	Q_3&=(1:0:1)\\
&Q_5=(1:1:1);	& 	P&=(0:0:1).
\end{align*}
Then we have the following.
\begin{align*}
&L_1\mbox{ is the line given by }x=0;\\
&L_2\mbox{ is the line given by }y=0;\\
&L_3\mbox{ is the line given by }x=y.
\end{align*}Since $L_4$ contains $P$, and is unequal to $L_1$ and $L_2$, there is an $m\in k^*$ such that $L_4$ is the line given by $my=x$. Since $Q_2,Q_7$ and $Q_8$ are not in $L_2$, and $Q_4$ is not in $L_1$, there are $a,b,c,u,v\in k$ such that 
\begin{align*}
&Q_2=(0:1:a);	& 	&Q_7=(m:1:v);\\
&Q_4=(1:0:b);	& 	&Q_8=(m:1:c).\\
&Q_6=(1:1:u);   &	&
\end{align*}

We define $\mathbb{A}^{6}$ to be the affine space with coordinate ring $T_{6}=k[a,b,c,m,u,v].$ Points in $\mathbb{A}^{6}$ correspond to configurations of the points $Q_1,\ldots,Q_8$. \\
Assume by contradiction that $C_{7,8}$, $C_{8,7}$, and $C_{6,5}$ all contain $P$. This assumption gives polynomial equations in the variables $a,b,c,m,u,v$, and hence defines an algebraic set $A_0$ in~$\mathbb{A}^{6}$. We define $S_0$ to be the algebraic set of all points in $\mathbb{A}^{6}$ that correspond to the configurations where three of the points $Q_1,\ldots,Q_8$ lie on a line, or six of the points lie on a conic. We want to show that $A_0$ is contained in $S_0$, which proves the proposition.\\
Note that the line containing $P$ and $Q_5$, which is $L_3$, does not contain any of the points $Q_1,Q_2,Q_3,Q_4,Q_8$. Set $(R_1,\ldots,R_7)=(Q_5,Q_1,Q_2,Q_3,Q_4,Q_8,P)$, then it follows from Lemma \ref{uniquecubic} that there is a unique cubic $D$ containing $Q_1,Q_2,Q_3,Q_4,Q_5,Q_8$ and $P$ that is singular in $Q_5$, and that $D$ does not contain $L_3$. By uniqueness,~$D$ must be equal to $C_{6,5}$, and therefore also contains $Q_7$. By Lemma \ref{genpos1}, the equation expressing that $Q_7$ is in $D$ (or equivalently that $P$ is in $C_{6,5}$) is given by det$(L)=0$, where $L$ is the matrix associated to $(R_1,\ldots,R_8)=(Q_1,Q_2,Q_3,Q_4,Q_7,Q_8,P,Q_5)$. We have 
$$\mbox{det}(L)=-m(m-1)(c-v)(b-1)(a-1)f,$$
where $f=\alpha v+\beta$ with $$\alpha=a-ac-bc+bm,\;\;\beta=b(a-1)m^2+b(c-2a)m+a(b+c-1).$$
The first five factors of det$(L)$ define subsets of $S_0$, and hence do not correspond to configurations where $Q_1,\ldots,Q_8$ are in general position. Therefore, $C_{6,5}$ contains~$P$ if and only if $f=0$.  
Define $V=Z(\alpha)$. Let $(a_0,b_0,c_0,m_0,u_0,v_0)$ be an element in $V\cap A_0$. Then $\alpha(a_0,b_0,c_0,m_0,u_0,v_0)=f(a_0,b_0,c_0,m_0,u_0,v_0)=0$, so we find $\beta(a_0,b_0,c_0,m_0,u_0,v_0)=0$. But $\alpha$ and $\beta$ do not depend on $v$, so this implies  $f(a_0,b_0,c_0,m_0,u_0,v')=0$ for every $v'$. So every element in $V\cap A_0$ corresponds to a configuration of $Q_1,\ldots,Q_8$ such that every point $(m:1:v')$ on $L_4$ is also contained in $D$. But if this is the case, then $D$ consists of $L_4$ and a conic, which is singular, since $Q_5$ is a singular point of $D$ that is not contained in~$L_4$. Since $L_4$ contains none of the points $Q_1,Q_2,Q_3,Q_4$, these four points are then on the singular conic, which implies that $Q_5$ is collinear with at least two other points. We conclude that $V\cap A_0$ is a subset of $S_0$.\\
Analogously, the fact that $C_{7,8}$ contains $P$ is expressed by det$(L')=0$, where $L'$ is the matrix denoted by $L$ in Lemma~\ref{genpos1} with $$(R_1,\ldots,R_8)=(Q_1,Q_2,Q_3,Q_4,Q_5,Q_6,P,Q_8).$$ 
We have 
$$\mbox{det}(L')=-m(u-1)(m-1)(b-1)(a-1)g,$$ where $g=\gamma u+\delta$ with 
$$\gamma=bm^3+(1-bc-c)m^2+(c^2-2c+1)m+a(1-c)+c^2-c,$$
and
\begin{multline}\delta =-abm^3+(abc+ab+ac-a+b-2bc)m^2+(ab-2abc+a+2bc^2-b-ac^2+2c^2-2c)m\\
+a(bc-b+2c^2-2c)-bc^2+bc-2c^3+2c^2.\nonumber
\end{multline}
The first five factors of det$(L')$ correspond to configurations where the eight points are not in general position, so $C_{7,8}$ contains $P$ if and only if $g=0$. Define $U=Z(\gamma)$. By the same reasoning as for $V\cap A_0$ (now using the fact that $D$ does not contain the line $L_3$), we have $U\cap A_0\subseteq S_0$.  
Set $$v'=\frac{-\beta}{\alpha}\;\;\mbox{ and }\;\;u'=\frac{-\delta}{\gamma}.$$ Define $\mathbb{A}^{4}$ to be the affine space with coordinate ring $T_{4}=k[m,a,b,c],$ and let $K_{4}$ be its fraction field. Let $W\subset\mathbb{A}^4$ be the set defined by $\alpha=\gamma=0$.  
Consider the ring homomorphism $\psi\colon T_{6}\longrightarrow K_{4}$ defined by
$$(m,a,b,c,u,v)\longmapsto(m,a,b,c,u',v').$$
This defines a morphism $i\colon\mathbb{A}^{4}\setminus W\longrightarrow \mathbb{A}^{6}\setminus \left(V\cup U\right)$, which is a section of the projection $\mathbb{A}^6\longrightarrow \mathbb{A}^4$ on the first four coordinates. Set $A_0'=A_0\setminus\left(V\cup U\right)$. Then we have $A_0\subset S_0$ if and only if $A_0'\subseteq S_0$. Moreover, $A_0'$ is contained in $Z(f,g)$, and since $f$ and $g$ are linear in $v$ and $u$ respectively, we have $i^{-1}(A_0')\cong A_0'$. Set $A_1=i^{-1}(A_0')$ and $S_1=i^{-1}(S_0)$, then $A_0'\subseteq S_0$ is equivalent to $A_1\subseteq S_1$.\\
Let $L''$ be the matrix denoted by $L$ in Lemma~\ref{genpos1} with $$(R_1,\ldots,R_8)=(Q_1,Q_2,Q_3,Q_4,Q_5,Q_6,P,Q_7).$$ Similarly to $C_{7,8}$, the fact that $C_{8,7}$ contains $P$ is expressed by the vanishing of the determinant of $L''$. We compute this determinant and write it in terms of the coordinates of $\mathbb{A}^4$ using $\psi$. We find the expression \begin{equation}\label{detl}-2abm(m-1)^3(b-1)(a-1)(a+b-1)f_1f_2f_3,
\end{equation}
with $$f_1=ac-a+bcm-bm^2-c^2+cm+c-m,$$

$$f_2=abm^2-2abm+ab-ac^2+2ac-a-bc^2+2bcm-bm^2,$$

and 

\begin{multline}f_3=abcm^2-2abcm+abc-abm^3+abm^2+abm-ab-ac^2m+2ac^2\\
+acm^2-3ac-am^2+am+a+2bc^2m-bc^2-3bcm^2+bc+bm^3\\
+bm^2-bm-2c^3+3c^2m+3c^2-cm^2-4cm-c+m^2+m. \nonumber\end{multline}

Expression (\ref{detl}) defines the set $A_1$ in $\mathbb{A}^4$. Since char $k\neq2$, we have $(\ref{detl})=0$ if and only if at least one of the non-constant factors of (\ref{detl}) equals zero. We show that all non-constant factors of expression (\ref{detl}) define components of $S_1$. If $a=0$, then $Q_2,\;Q_3$ and $Q_5$ are contained in the line given by $x-z=0$. Similarly, $b=0$ implies that $Q_1,\;Q_4$ and $Q_5$ are on the line given by $y-z=0$, and $a+b-1=0$ implies that $Q_2,\;Q_4,$ and $Q_5$ are on the line given by $bx+ay-z=0$. If $m=0$ then $L_4=L_2$, and $m=1$ implies $L_4=L_3$, so in both cases there are four points on a line. If $a=1$ or $b=1$, then two of the eight points would be the same. Set $(R_1,\ldots,R_6)=(Q_3,\ldots,Q_8)$, and let~$N$ be the corresponding matrix from Lemma~\ref{genpos1}. We compute the determinant of $N$ and find that $f_1f_2f_3$ divides det$(N)$.
This means that $f_1$, $f_2$, as well as $f_3$ define components of $S_1$, more specifically, they define configurations where $Q_3,\ldots,Q_8$ are on a conic.
We conclude that all irreducible components of $A_1$ are contained in~$S_1$, which finishes the proof.\end{proof}

\begin{remark}\label{groebner1}Note that in theory we could have proved Proposition \ref{key1} with a computer, by checking that $A_0$ is contained in $S_0$ using Gr\"{o}bner bases. However, in practice, this turned out to be too big for \texttt{magma} to do. 
\end{remark}

We can now prove Theorem \ref{thm1}. Recall that a maximal set of exceptional curves that are concurrent in a point on the ramification curve consists of curves and their partners (Remark \ref{partners}). 

\begin{proofoftheorem1} First note that by Proposition \ref{bovengrens}, the number of exceptional curves through any point in $X$ is at most sixteen in all characteristics; this proves the case char $k=2$. \\
Now assume char $k\neq2$. Consider the clique $K=\{e_1,\ldots,e_6,e_1',\ldots,e_6'\}$ in $G$, where 
\begin{align*}
&e_1=L-E_1-E_2;\\
&e_2=L-E_3-E_4;\\
&e_3=L-E_5-E_6;\\
&e_4=L-E_7-E_8;\\
&e_5=3L-E_1-E_2-E_3-E_4-E_5-E_6-2E_8;\\
&e_6=3L-E_1-E_2-E_3-E_4-2E_5-E_7-E_8,
\end{align*}
and $e_i'$ is the partner of $e_i$, for all $i\in\{1,\ldots,6\}$. Note that $K$ is an element of $K_6$ as defined in Remark \ref{partners}. By Remark \ref{configuratie}, the classes $e_1,\ldots,e_4$ correspond to the strict transforms of the four lines through $P_i$ and $P_{i+1}$ for $i\in\{1,3,5,7\}$, and $e_5,e_6,e_5'$ correspond to the strict transforms of the unique cubics through the points $P_1,\ldots,P_6,P_8$, and the points $P_1,\ldots,P_5,P_7,P_8$, and the points $P_1,\ldots,P_6,P_7$, respectively, that are singular in $P_8$, and $P_5$, and $P_7$, respectively. \\
Now let $K'$ be a clique in $G$ with only edges of weights 1 and 3, consisting of at least six sets of an exceptional class with its partner. Let $\{\{f_1,f_1'\},\ldots,\{f_6,f_6'\}\}$ be a set of six such sets in $K'$. Since $W$ acts transitively on $K_6$ by Proposition \ref{transitief}, after changing the indices and interchanging $f_i$'s with their partner if necessary, there is an element $w\in W$ such that $f_i=w(e_i)$ and $f_i'=w(e_i')$ for $i\in\{1,\ldots,6\}$. For $i\in\{1,\ldots,8\}$, set $E_i'=w(E_i)$. Since the $E_i'$ are pairwise disjoint, by Lemma \ref{BD} we can blow down $E_1',\ldots,E_8'$ to points $Q_1,\ldots,Q_8\in \mathbb{P}^2$ that are in general position, such that $X$ is isomorphic to the blow-up of~$\mathbb{P}^2$ at $Q_1,\ldots,Q_8$, and $E_i'$ is the class in Pic $X$ corresponding to the exceptional curve above $Q_i$ for all~$i$. By Remark \ref{bijectionblowdown}, the sequence $(E_1',\ldots,E_8')$ induces a bijection between the exceptional curves on $X$ and the 240 vectors in Proposition \ref{exceptional}, such that the element $f_i$ corresponds to the class of the strict transform of the line through $Q_{2i-1}$ and $Q_i$ for $i\in\{1,\ldots,4\}$, the elements $f_5$ and $f_6$ correspond to the classes of the strict transforms of the unique cubics through the points $Q_1,\ldots,Q_6,Q_8$ and $Q_1,\ldots,Q_5,Q_7,Q_8$, respectively, that are singular in $Q_8$ and $Q_5$ respectively, and $f_i'$ is the unique class in $C$ intersecting $f_i$ with multiplicity three for all $i$. From Proposition~\ref{key1} it follows that the curves on $X$ corresponding to $f_1,\ldots,f_6,f_5'$ and $f_6'$ are not concurrent.\\
We conclude that a set of at least six exceptional curves and their partners is never concurrent. Since any maximal set of exceptional curves going through the same point on the ramification curve forms a clique consisting of curves and their partners, hence of even size, we conclude that this maximum is at most ten.  \qed
\end{proofoftheorem1}

\section{Proof of Theorem \ref{thm2}}\label{poffthecurve}

In this section we prove Theorem \ref{thm2}. The structure of the proof is similar to that of Theorem \ref{thm1}; we first determine the cliques in $G$ that possibly come from a set of exceptional curves that are concurrent outside the ramification curve of~$\varphi$ (Remark~\ref{vormcliques}), and show that their maximal size is 12 (Proposition \ref{max12bvt}). Then we show that the group $W$ acts transitively on these cliques of size~12 (Proposition \ref{trans12}) and 11 (Proposition \ref{trans11}), and finally we show that ten specific curves on $X$ are not concurrent in Proposition~\ref{key2}. This final proposition is again key to the proof of Theorem \ref{thm2}. 

\begin{remark}\label{vormcliques}
From Remark \ref{rem} we know that cliques in $G$ corresponding to exceptional curves that intersect each other in a point outside the ramification curve have no edges of weight 3. We conclude that these cliques contain only edges of weights 1 and 2. 
\end{remark}

\begin{proposition}\label{max12bvt}
The maximal size of cliques in $G$ with only edges of weights 1 and 2 is 12, and there are no inclusion-wise maximal cliques of size 11 with only edges of weights 1 and 2. 
\end{proposition}
\begin{proof}
We use the correspondence with the graph $\Gamma$ in \cite{WvL}, where the corresponding cliques have only edges of colors $-1$ and 0; the statement is in Proposition~5.20~(iii).
\end{proof}

\begin{proposition}\label{trans12}
The group $W$ acts transitively on the set of cliques of size 12 in~$G$ with only edges of weights 1 and 2. 
\end{proposition}
\begin{proof}
This is \cite{WvL}, Proposition 5.21 (i).
\end{proof}

\begin{proposition}\label{trans11}
The group $W$ acts transitively on the set of cliques of size 11 in~$G$ with only edges of weights 1 and 2. 
\end{proposition}
\begin{proof}
By Proposition \ref{max12bvt}, any clique of size 11 with only edges of weights 1 and 2 is contained in a clique of size 12 with only edges of weights 1 and 2. By Corollary~5.22 in \cite{WvL}, for such a clique $K$ of size 12, the stabilizer $W_K$ acts transitively on~$K$, which implies that $W_K$ also acts transitively on the set of cliques of size 11 within~$K$. Since $W$ acts transitively on the set of all cliques of size 12 with only edges of weights~1 and 2 by Proposition \ref{trans12}, the statement follows. 
\end{proof}

Now that we know which cliques in $G$ to look at and what their maximal size is, we show that ten specific exceptional curves on $X$ are not concurrent in Proposition~\ref{key2}.  

\begin{remark}\label{uniquedingen}It is well known that two distinct points in $\mathbb{P}^2$ define a unique line, and five points in $\mathbb{P}^2$ in general position define a unique conic. Now let $R_1,\ldots,R_8$ be eight distinct points in $\mathbb{P}^2$ in general position. The linear system $\mathcal{Q}$ of quartics in~$\mathbb{P}^2$ has dimension $14$. For $R_i,R_j,R_l\in\{R_1,\ldots,R_8\}$, requiring a quartic to contain $R_1,\ldots,R_8$ and be singular in in $R_i,R_j,R_l$ gives $8+3\cdot2=14$ linear relations. Since the eight points are in general position, the 14 linear conditions are linearly independent, so this gives a zero-dimensional linear subsystem of $\mathcal{Q}$. Hence there is a unique quartic containing all eight points that is singular in $R_i,R_j,R_l$. 
\end{remark}

Let $R_1,\ldots,R_8$ be eight points in $\mathbb{P}^2$ in general position. Remark \ref{uniquedingen} allows us to define the following curves.
\begin{align*}
&L_1\mbox{ is the line through }R_1\mbox{ and }R_2;\\
&L_2\mbox{ is the line through }R_3\mbox{ and }R_4;\\
&C_1\mbox{ is the conic through }R_1,\;R_3,\;R_5,\;R_6\mbox{ and }R_7;\\
&C_2\mbox{ is the conic through }R_1,\;R_4,\;R_5,\;R_6\mbox{ and }R_8;\\
&C_3\mbox{ is the conic through }R_2,\;R_3,\;R_5,\;R_7\mbox{ and }R_8;\\
&C_4\mbox{ is the conic through }R_2,\;R_4,\;R_6,\;R_7\mbox{ and }R_8;\\
&D_1\mbox{ is the quartic through all eight points with singular points in }R_1,\;R_7\mbox{ and }R_8;\\
&D_2\mbox{ is the quartic through all eight points with singular points in }R_2,\;R_5\mbox{ and }R_6;\\
&D_3\mbox{ is the quartic through all eight points with singular points in }R_3,\;R_6\mbox{ and }R_8;\\
&D_4\mbox{ is the quartic through all eight points with singular points in }R_4,\;R_5\mbox{ and }R_7.
\end{align*}

\begin{proposition}\label{key2}
Assume that the characteristic of $k$ is not 3. Then the ten curves $L_1,\;L_2,\;C_1,\ldots C_4,\;D_1,\ldots,D_4$ are not concurrent.
\end{proposition}

\begin{remark}\label{groebner2}
As in the case of Proposition \ref{key1}, in theory we could prove Proposition~\ref{key2} with a computer by using Gr\"{o}bner bases, but in practice, this is undoable since the computations become too big (see also Remark \ref{groebner1}). In the case of Proposition \ref{key2} the computations become even bigger, since we now have 10 curves to check, of which four are of degree 4, in contrast to the 7 curves of degrees at most~3 in Proposition \ref{key1}.  
\end{remark}

Before we write down the proof of Proposition \ref{key2}, we make some reductions. In $\mathbb{P}^2$, we can choose four points in general position. Fix these and call them $Q_1,Q_5,Q_6,$ and $X$. We are interested in those configurations of five points $Q_2,Q_3,Q_4,Q_7$ and $Q_8$ in $\mathbb{P}^2$ such that the following 11 conditions hold.
\begin{align*}
0) &\mbox{ The points }Q_1,\ldots,Q_8\mbox{ are in general position.}\\
1) &\mbox{ There is a line through }X,Q_1,Q_2.\\
2) &\mbox{ There is a line through }X,Q_3,Q_4.\\
3) &\mbox{ There is a conic through }X,Q_1,Q_3,Q_5,Q_6,Q_7.\\
4) &\mbox{ There is a conic through }X,Q_1,Q_4,Q_5,Q_6,Q_8.\\
5) &\mbox{ There is a conic through }X,Q_2,Q_3,Q_5,Q_7,Q_8.\\
6) &\mbox{ There is a conic through }X,Q_2,Q_4,Q_6,Q_7,Q_8.\\
7) &\mbox{ There is a quartic through all nine points that is singular in }Q_1,Q_7,Q_8.\\
8) &\mbox{ There is a quartic through all nine points that is singular in }Q_2,Q_5,Q_6.\\
9) &\mbox{ There is a quartic through all nine points that is singular in }Q_3,Q_6,Q_8.\\
10) &\mbox{ There is a quartic through all nine points that is singular in }Q_4,Q_5,Q_7.
\end{align*}
 
We will prove Proposition \ref{key2} by showing that there are no such configurations: all of the configurations satisfying 1-10 are such that condition 0 is violated. 

\vspace{11pt}

We consider the space $(\mathbb{P}^2)^5$. Within this space, we define the following two sets.
$$Y=\left\{\left(Q_2,Q_3,Q_4,Q_7,Q_8\right)\in(\mathbb{P}^2)^5\;\left|\;\mbox{conditions 1-5 are satisfied}\right.\right\}.$$
$$S=\left\{\left(Q_2,Q_3,Q_4,Q_7,Q_8\right)\in(\mathbb{P}^2)^5\;\left|\;\mbox{three of the points }Q_1,\ldots,Q_8\mbox{ are collinear}\right.\right\}.$$

\vspace{11pt}

Note that for an element $(Q_2,Q_3,Q_4,Q_7,Q_8)$ in $S$, condition 0 is violated. Let~$F_1$ be the linear system of conics through $X,Q_1,Q_5,Q_6$. Note that this is a one-dimensional linear system that is isomorphic to $\mathbb{P}^1$. Let $F_2$ be the linear system of lines through $X$, which is also isomorphic to $\mathbb{P}^1$. We will show in Proposition \ref{inversephi} that there is a bijection between $Y\setminus S$ and a subset of $F_1^2\times F_2^3$. We start with two lemmas.

\begin{lemma}\label{QP}
If $(Q_2,Q_3,Q_4,Q_7,Q_8)$ is a point in $Y\setminus S$, then we have $Q_i\neq X$ for $i=2,3,4,7,8$.
\end{lemma}
\begin{proof}
Take a point $Q=(Q_2,Q_3,Q_4,Q_7,Q_8)$ in $Y\setminus S$. Since $Q$ is an element of~$Y$, by condition 1 the points $X,Q_1,Q_2$ are on a line. That means that if $X=Q_i$ for $i=3,4,7,8$, the points $Q_i,Q_1,Q_2$ would be on a line, contradicting the fact that~$Q$ is not in $S$. Moreover, by condition 2, the points $X,Q_3,Q_4$ are on a line, so if $X=Q_2$ then $Q_2,Q_3,Q_4$ are on a line, again contradicting the fact that $Q$ is not in~$S$.
\end{proof}

The following result is well known, but we include a proof, as we could not find a reference for this exact statement.

\begin{lemma}\label{uniqueC}
If $S_1,\ldots,S_5$ are five distinct points in $\mathbb{P}^2$, such that $S_1,\ldots,S_4$ are in general position, then there is a unique conic containing $S_1,\ldots,S_5$. This conic is irreducible if all five points are in general position.
\end{lemma}
\begin{proof}
The linear system of conics containing $S_1,\ldots,S_4$ is one-dimensional and has only these four points as base points. Requiring for a conic in this linear system to contain the point $S_5$ gives a linear condition, and since $S_5$ is different from $S_1,\ldots,S_4$, this condition defines a linear subspace of dimension at least zero. If there were two distinct conics in this subspace, they would intersect in 5 distinct points, so they would have a common component, which is a line. Since no 4 of the points $S_1,\ldots, S_5$ are collinear, there are at most 3 of the 5 points on this line. But then the other two points uniquely determine the second component of both conics, contradicting that they are distinct. We conclude that there is a unique conic containing $S_1,\ldots, S_5$. If, moreover, $S_5$ is such that all five points are in general position, then no three of them are collinear by definition, so the unique conic containing them cannot contain a line, hence it is irreducible. 
\end{proof}

Let $(Q_2,Q_3,Q_4,Q_7,Q_8)$ be a point in $Y\setminus S$. Note that by condition 3, there is a conic through the points $X,Q_1,Q_3,Q_5,Q_6$, and $Q_7$, and by Lemma \ref{uniqueC} it is unique, since $X, Q_1, Q_5, Q_6$ are in general position. We call this conic $A_1$. By the same reasoning and condition 4, there is a unique conic containing the points $X, Q_1,Q_4,Q_5,Q_6,Q_8$. We call this conic $A_2$. 
By Lemma \ref{QP}, the points $Q_3,Q_7,Q_8$ are all different from $X$, so we can define the line $M_1$ through $X$ and $Q_3$, the line $M_2$ through $X$ and $Q_7$, and the line $M_3$ through $X$ and $Q_8$. 

\vspace{11pt}

We now define a map
\begin{align*}\varphi\colon Y\setminus S&\longrightarrow F_1^2\times F_2^3,\\
(Q_2,Q_3,Q_4,Q_7,Q_8)&\longmapsto(A_1,A_2,M_1,M_2,M_3).
\end{align*}

Note that $\varphi$ is well defined by the definitions of $A_1,A_2,M_1,M_2,M_3$. We want to describe its image. To this end, define the set 

$$U=\left\{(B_1,B_2,N_1,N_2,N_3)\in F_1^2\times F_2^3\;\left|\begin{array}{c}
B_1,B_2\mbox{ irreducible}\\
B_1\neq B_2\\
N_1,N_2\mbox{ not tangent to }B_1\\
N_1,N_3\mbox{ not tangent to }B_2\\
N_1\neq N_2,N_3\\
Q_1,Q_5,Q_6\not\in N_1,N_2,N_3\\\end{array}\right.\right\}$$

\vspace{11pt}

\begin{lemma}\label{imagephi}
The image of $\varphi$ is contained in $U$.
\end{lemma}
\begin{proof}
Take a point $Q=(Q_2,Q_3,Q_4,Q_7,Q_8)\in Y\setminus S$ and consider its image under $\varphi$ given by  $\varphi(Q)=(A_1,A_2,M_1,M_2,M_3)$. Since $Q$ is not in $S$, by Lemma \ref{uniqueC}, the conics $A_1$  and $A_2$ are unique and irreducible. Moreover, if they were equal to each other, then they would both contain the points $X,Q_3,Q_4$, which are collinear by condition 2, contradicting the fact that they are irreducible. \\
The line $M_1$ is tangent to $A_1$ only if $X$ is equal to $Q_3$, the line $M_2$ is tangent to $A_1$ only if $X$ is equal to $Q_7$, and the line $M_3$ is tangent to $A_2$ only if $X$ is equal to $Q_8$, all of which are impossible by Lemma \ref{QP}. Note that by condition 2, the line $M_1$ contains $Q_4$, so $M_1$ is tangent to $A_2$ only if $X=Q_4$, which is again impossible by Lemma \ref{QP}. If $M_2$ or $M_3$ were equal to $M_1$, then either $Q_7$ or $Q_8$ is contained in $M_1$, which also contains the points $X,Q_3,Q_4$. But this can not be true since $Q$ is not in $S$. If $M_1$ or $M_2$ contained any of the points $Q_1,Q_5,Q_6$, then this line would have three points in common with $A_1$, which implies that $A_1$ contains a line, contradicting the fact that $A_1$ is irreducible. Similarly, if $M_3$ contained $Q_1,Q_5,$ or $Q_6$, then $A_2$ would contain $M_3$, contradicting the irreducibility of $A_2$. 
\end{proof}

We want to define an inverse to $\varphi$. Let $A=(A_1,A_2,M_1,M_2,M_3)$ be a point in~$U$. Since the conics $A_1$ and $A_2$ are irreducible, they do not contain any of the lines $M_1,M_2,M_3$, and moreover, since $M_1,M_2$ are not tangent to $A_1$, and $M_1,M_3$ are not tangent to $A_2$, we can define the following five points in $\mathbb{P}^2$.
\begin{align*}
Q_3=&\mbox{ the point of intersection of }A_1\mbox{ with }M_1\mbox{ that is not }X.\\
Q_4=&\mbox{ the point of intersection of }A_2\mbox{ with }M_1\mbox{ that is not }X.\\
Q_7=&\mbox{ the point of intersection of }A_1\mbox{ with }M_2\mbox{ that is not }X.\\
Q_8=&\mbox{ the point of intersection of }A_2\mbox{ with }M_3\mbox{ that is not }X.
\end{align*}

\begin{lemma}\label{uniqueconic}There is a unique conic through $X, Q_3, Q_5, Q_7$, and $Q_8$, which does not contain the line through $X$ and $Q_1$. 
\end{lemma}
\begin{proof}Note that $Q_3$ and $Q_7$ are different from $X$ by definition, and they are different from $Q_1,Q_5,Q_6$ since $Q_1,Q_5,Q_6$ are not contained in $M_1$, nor in $M_2$, by definition of $U$. If $Q_3$ were equal to $Q_7$, then $M_1$ and $M_2$ would both contain $X$ and $Q_3$, hence they would be equal, contradicting the fact that $A$ is an element of $U$. So $X, Q_3, Q_5, Q_7$ are all distinct, and since they are all contained in $A_1$, they are in general position because $A_1$ is irreducible. We will show that $Q_8$ is different from any of these four points. By definition, $Q_8$ is different from $X$. If $Q_8$ were equal to $Q_3$, then $A_1$ and $A_2$ would both contain $X,Q_1,Q_5,Q_6$ and $Q_3$. But since $Q_3$ is different from $X,Q_1,Q_5,Q_6$, there is a unique conic through these five points by Lemma \ref{uniqueC}. So this would imply $A_1=A_2$, contradicting the fact that $A$ is in $U$. Hence $Q_8$ is different from $Q_3$, and similarly, $Q_8$ is different from $Q_7$. Finally, $Q_8$ is different from $Q_5$, since the line $M_3$ does not contain $Q_5$. We conclude that by Lemma \ref{uniqueC}, there is a unique conic $C$ through the points $X,Q_3,Q_5,Q_7,$ and~$Q_8$.
Note that $X, Q_3, Q_5, Q_7$ are all distinct from $Q_1$. If $C$ contained the line $L$ through $X$ and $Q_1$, then $C$ would be the union of two lines (one of them being $L$). This means that either $L$ would contain one of the points $Q_3,Q_5,Q_7$, or the points $Q_3,Q_5,Q_7$ are all on the second line. But since $X,Q_1,Q_3,Q_5,Q_7$ are all in $A_1$, which is irreducible, both of these cases would be a contradiction. We conclude that $C$ does not contain~$L$.
\end{proof}

We now define a fifth point $Q_2$ to be the point of intersection of the conic through $X, Q_3, Q_5, Q_7,Q_8$ with the line through $X$ and $Q_1,$ that is not $X$. Note that $Q_2$ is well defined by Lemma \ref{uniqueconic}. From a point $A$ in $U$ we have now defined an element $(Q_2,Q_3,Q_4,Q_7,Q_8)$ of $(\mathbb{P}^2)^5$, and it is easy to see that for this point conditions 1-5 are satisfied, hence it is an element of $Y$. This leads us to define the following map. 
\begin{align*}
\psi\colon U&\longrightarrow Y,\\
(A_1,A_2,M_1,M_2,M_3)&\longmapsto(Q_2,Q_3,Q_4,Q_7,Q_8).
\end{align*}

Let $T$ be the set $\psi^{-1}(S)$. 

\begin{proposition}\label{inversephi} The map $\psi|_{U\setminus T}\colon U\setminus T\longrightarrow Y\setminus S$ is a bijection, with inverse given by $\varphi$.
\end{proposition}
\begin{proof}
Take $A=(A_1,A_2,M_1,M_2,M_3)\in U\setminus T$. Write $\psi(A)=(Q_2,Q_3,Q_4,Q_7,Q_8)$ and $\varphi(\psi(A))=(B_1,B_2,N_1,N_2,N_3)$. Since $\psi(A)$ is not in $S$ by definition of $T$, no three of the points $Q_1,\ldots,Q_8$ are collinear. Therefore, $B_1$ and $B_2$ are the unique and irreducible conics through $Q_1,Q_3,Q_5,Q_6,Q_7$ and through $Q_1,Q_4,Q_5,Q_6,Q_8$, respectively, by Lemma \ref{uniqueC}. Since $A_1$ and $A_2$ both contain $Q_1,Q_5,Q_6$, and $A_1$ contains $Q_3$ and $Q_7$ and $A_2$ contains $Q_4$ and $Q_8$ by definition of $\psi(A)$, we conclude that $B_1=A_1$ and $B_2=A_2$. The line $N_1$ is defined as the line containing $X$ and $Q_3$, which are both contained in $M_1$ as well by definition. We conclude that $N_1=M_1$, and similarly $N_2=M_2$, and $N_3=M_3$. We conclude that $\varphi(\psi(A))=A$. This proves injectivity of $\psi|_{U\setminus T}$.\\
To prove surjectivity of $\psi|_{U\setminus T}$, take a point $Q=(Q_2,Q_3,Q_4,Q_7,Q_8)$ in $Y\setminus S$. Write $\varphi(Q)=(A_1,A_2,M_1,M_2,M_3)$ and $\psi(A_1,A_2,M_1,M_2,M_3)=(R_2,R_3,R_4,R_7,R_8)$.
The point $R_3$ is defined by taking the second point of intersection of $A_1$ with the line $M_1$ through $X$ and $Q_3$. Since $A_1$ is irreducible ($\varphi(Q)$ is in $U$ by Lemma \ref{imagephi}), it does not contain $M_1$, so $R_3=Q_3$. Similarly, we have $R_7=Q_7$, $R_4=Q_4$, and $R_8=Q_8$. Therefore there is a unique conic~$C$ through $X,Q_3,Q_5,Q_7,Q_8$ by Lemma~\ref{uniqueconic}. Since there is a conic through $X,Q_3,Q_5,Q_7,Q_8$ and $Q_2$ by condition 5, we conclude that $C$ contains $Q_2$ by uniqueness. Since the line $L$ through $X$ and $Q_1$ is not contained in $C$ by Lemma~\ref{uniqueconic}, and since $L$ contains $Q_2$ by condition~1, it follows that $Q_2$ is the second point of intersection of $L$ and $C$. Hence $R_2=Q_2$. We conclude that $\psi(\varphi(Q))=Q$, and hence $\varphi(Q)$ is contained in $U\setminus T$, and $\psi|_{U\setminus T}$ is surjective. \\
Since $\psi_{U\setminus T}\colon U\setminus T\longrightarrow Y\setminus S$ is a bijection and we showed that for all elements $A\in U\setminus T$ we have $\varphi(\psi(A))=A$, we conclude that $\varphi$ is the inverse function.  
\end{proof}

We now prove Proposition \ref{key2}. The computations were verified in \texttt{magma}; see \cite{Magmab} for the code. 

\begin{proofofkey2}Recall the curves $L_1,L_2,C_1,\ldots,C_4,D_1,\ldots,D_4$ that are defined above Proposition \ref{key2}. We assume that these ten curves contain a common point $P$, and will show that this contradicts the fact that $R_1,\ldots,R_8$ are in general position. First note that if $P$ were equal to one of the eight points $R_1,\ldots,R_8$, then one of the conics would contain six of the eight points, which would contradict the fact that $R_1,\ldots,R_8$ are in general position. Moreover, if $P$ and any two of the three points $R_1,R_5,R_6$ lie on a line $L$, then the conic $C_1$ would intersect~$L$ in $P$ and the two points. But this implies that $C_1$ is not irreducible, and since $C_1$ contains five of the points $R_1,\ldots,R_8$, this implies that at least three of them are collinear, contradicting the fact that $R_1,\ldots,R_8$ are in general position. We conclude that $R_1,R_5,R_6$ and $P$ are in general position. \\
Let $(x:y:z)$ be the coordinates in $\mathbb{P}^2$. Without loss of generality, after applying an automorphism of $\mathbb{P}^2$ if necessary, we can choose $R_1,R_5,R_6$, and $P$ to be any four points in general position in $\mathbb{P}^2$. We now distinguish between char $k\neq2$ and char~$k=2$.

\vspace{5pt}

\underline{Assume char $k\neq 2$.} Set
\begin{align*}
&R_1=(1:0:1); &	R_6&=(0:-1:1);\\
&R_5=(0:1:1); &	P&=(-1:0:1).
\end{align*} 
It follows that the line $L_1$, which contains $R_1$ and $P$, is given by $y=0$. The linear system of quadrics through $R_1,\;R_5,\;R_6$ and $P$ is generated by two linearly independent quadrics, and we take these to be $x^2+y^2-z^2$ and $xy$. Let $l,m \in k$ be such that\begin{align*}
&C_1\mbox{ is given by }x^2+y^2-z^2=2lxy;\\
&C_2\mbox{ is given by }x^2+y^2-z^2=2mxy.
\end{align*}
Since $R_3,R_4,R_7,$ and $R_8$ are not contained in $L_1$, there are $s,t,u\in k$ such that
\begin{align*}
&\mbox{the line }L_2\mbox{ is given by }sy=x+z;\\
&\mbox{the line }L_3\mbox{ through }P\mbox{ and }R_7\mbox{ is given by }ty=x+z;\\
&\mbox{the line $L_4$ through }P\mbox{ and }R_8\mbox{ is given by }uy=x+z.
\end{align*}
We want to show that all possible configurations of the five points $R_2,R_3,R_4,R_7,R_8$ in~$\mathbb{P}^2$ such that all ten curves contain $P$, are such that $R_1,\ldots,R_8$ are not in general position. By Proposition \ref{inversephi}, all configurations of $R_2,R_3,R_4,R_7,R_8$ such that $L_1,L_2,C_1,C_2,C_3$ contain the point $P$ and no three of the points $R_1,\ldots,R_8$ are collinear are given in terms of the conics $C_1$ and $C_2$ and the lines $L_2,L_3,L_4$. By computing the appropriate intersections we find
\begin{align*}
&R_3=\left(-s^2+1:2l-2s:2ls-s^2-1\right);\\
&R_4=\left(-s^2+1:2m-2s:2ms-s^2-1\right);\\
&R_7=\left(-t^2+1:2l-2t:2lt-t^2-1\right);\\
&R_8=\left(-u^2+1:2m-2u:2mu-u^2-1\right).
\end{align*}
By Lemma \ref{uniqueconic}, there is a unique conic containing $R_3,R_5,R_7,R_8$, and $P$, and we compute a defining polynomial and find 
\begin{multline}\left(2l^2u+2l^2-2lmu-2lm-lsu-ls-ltu-lt+lu^2+2lu+l+mst+ms+mt\right.\nonumber\\
\left.-2mu-m+st-su-tu+u^2
\vphantom{\tfrac12}\right)x^2\nonumber\\
\left(2l^2u^2+2l^2u+2lmst-2lmsu-2lmtu-2lmu-lstu+lst-lsu+ls-ltu+lt\right.\nonumber\\
\left.+2lu^2+lu+l+mstu+mst-msu-ms-mtu-mt-mu-m\vphantom{\tfrac12}\right)xy\nonumber \\ 
+2(u+1)(l+1)(l-m)xz\nonumber\\
+\left(\vphantom{\tfrac12}lstu+lst+lu^2+lu-mstu-msu-mtu-mu+st-su-tu+u^2\vphantom{\tfrac12}\right)y^2\nonumber\\
+(u+1)(t+1)(s+1)(l-m)yz\nonumber\\
+\left(\vphantom{\tfrac12}lsu+ls+ltu+lt-lu^2+l-mst-ms-mt-m-st+su+tu-u^2\vphantom{\tfrac12}\right)z^2.
\end{multline} 

Intersecting this conic with the line $L_1$ gives besides $P$ the point $R_2$, and we find 
\begin{multline}R_2=(-(lsu+ls+ltu+lt-lu^2+l-mst-ms-mt-m-st+su+tu-u^2)\nonumber\\
:0:\nonumber\\
(2l^2-2lm-ls-lt)(u+1)+lu^2+2lu+l+mst+ms+mt-2mu-m+st-su-tu+u^2).\end{multline}

We define $\mathbb{A}^{5}$ to be the affine space with coordinate ring $T_{5}=k[l,m,s,t,u].$ Following all the above, points in $\mathbb{A}^{5}$ correspond to configurations of the points $R_1,\ldots,R_8$. The fact that the ten curves contain $P$ gives polynomial equations in these five variables, and hence defines an algebraic set $A_0$ in $\mathbb{A}^{5}$. We define $S_0$ to be the algebraic set of all points in $\mathbb{A}^{5}$ that correspond to the configurations where the points $R_1,\ldots,R_8$ are not in general position. We want to show that $A_0$ is contained in $S_0$, which would prove the proposition. In what follows we will show that indeed every component of~$A_0$ is contained in $S_0$.\\
Note that by construction of $R_1,\ldots,R_8$, the curves $L_1,L_2,C_1,C_2,C_3$ contain~$P$. We will add conditions for $C_4,D_1,\ldots,D_4$ to contain $P$, too. We start with $C_4$. The equation expressing that $P$ is contained in $C_4$, is given by det($N)=0$, where $N$ is the matrix in Lemma \ref{genpos1} corresponding to $(R_2,R_4,R_6,R_7,R_8,P)$. This determinant is given by 
$$\mbox{det}(N)=16(u+1)(t+1)(s+1)(s-u)(m-u)(m-s)(l-t)(l-m)f_1f_2,$$
where 
\begin{multline}f_1=l^2u+l^2-lmu-lm-lsu-ls-ltu-lt+lu^2+lu+mst+ms\nonumber\\
+mt-mu+st-su-tu+u^2
,\end{multline}
and 
$$f_2=at^2+btu+cu^2+dt+eu+f,$$
with \begin{align*}
&a=(s+1)(m-1)(m+1),\;\;&\;b=d=-e=2s(m-1)(l+1),\\
&c=(s-1)(l-1)(l+1),\;\;&\;f=(l-m)(ls-l-ms-m+2s).\end{align*}
Let $F_2\subset\mathbb{A}^5$ be the affine variety given by $f_2=0$. Every component of $A_0$ is contained in one of the components of the algebraic set given by det$(N)=0$. With \texttt{magma} it is an easy check that apart from $f_2$, all non-constant factors of det$(N)$ define configurations of $R_1,\ldots,R_8$ where three of the points are collinear (see \cite{Magmab}; $f_1=0$ corresponds to $R_2,R_3,R_4$ being collinear), and hence they define components of~$S_0$. Therefore, it suffices to prove that $A_0\cap F_2$ is contained in $S_0$. \\
Since $f_2$ is quadratic in $t$ and $u$, the projection $\pi$ from $F_2$ to the affine space $\mathbb{A}^3$ with coordinates $l,m,s$ has fibers that are (possibly non-integral) affine conics. Let~$\Delta$ be the discriminant of the quadratic form that is the homogenisation of $f_2$ with respect to $t$ and $u$, which is given by $$\Delta=4acf-ae^2-b^2f+bde-cd^2;$$ the singular fibers of $\pi$ lie exactly above the points $(l,m,s)\in\mathbb{A}^3$ for which $\Delta=0$. We compute the factorization of $\Delta$ in $\mathbb{Z}[l,m,s]$, and find $$\Delta=4(s-1)(s+1)(m-1)(m+1)(l-1)(l+1)(l-m)g,$$
with $g=ls-l-ms-m+2s$. All non-constant factors of $\Delta$ except for $g$, when viewed as elements of $T_5$, define components of $S_0$ in $\mathbb{A}^5$. Therefore, the fibers under~$\pi$ above the zero sets of these factors in $\mathbb{A}^3$ are contained in $S_0$. We will show that the same holds for the inverse image under $\pi$ of the zero set $Z(g)\subset\mathbb{A}^3$ of $g$, which is given by the zero set $Z(f_2,g)$ in $\mathbb{A}^5$. Note that we can write $$f_2=(s-1)(l+1)(u-t)a_1+(t-1)ga_2,$$
with $a_1=(l-1)(u+1) - (m+1)(t-1)$ and $a_2=(l+1)(u+1) - (m+1)(t+1)$. Therefore, the set $Z(f_2,g)$ is given by $g=(s-1)(l+1)(u-t)a_1=0$, so $Z(f_2,g)$ is the union of four algebraic sets: $$Z(f_2,g)=Z(g,s-1)\cup Z(g,l+1)\cup Z(g,u-t)\cup Z(g,a_1)\subset \mathbb{A}^5.$$ Note that $s-1,l+1,$ and $u-t$ define components of $S_0$, so the first three terms in this union are contained in $S_0$. With \texttt{magma}, we check that the irreducible polynomial $\gamma=(m-u)(l-1)g + (l-s)(m-1)a_1$ corresponds to a configuration where the six points $R_3,\ldots,R_8$ are contained in a conic, and hence it defines a component of $S_0$. Since $\gamma$ is contained in the ideal in $\mathbb{Z}[l,m,s,t,u]$ generated by $g$ and $a_1$, it follows that $Z(g,a_1)$ is also contained in $S_0$. We conclude that all the singular fibers of~$\pi$ lie in $S_0$. \\
The generic fiber $F_{2,\eta}$ of $\pi$ is a conic in the affine plane $\mathbb{A}^2$ with coordinates $t$ and~$u$ over the function field $k(l,m,s)$, where $l,m,s$ are transcendentals. This fiber contains the point $(t,u)=(l,m)$. We can parametrize $F_{2,\eta}$ with a parameter $v$ by intersecting it with the line $M$ given by $v(t-l)=(u-m)$, which intersects $F_{2,\eta}$ in the point $(l,m)$ and a second intersection point that we associate to $v$. Consider the open subset $F_2'\subset F_2$ given by the complement in $F_2$ of the singular fibers of $\pi$ and the hyperplane section defined by $t-l=0$, so $F_2\setminus F_2'\subset S_0$. 
In what follows, we use the idea of this parametrization to construct an isomorphism between $F_2'$ and an open subset of the affine space $\mathbb{A}^4$ with coordinates $l,m,s,v$. \\
Consider the ring $T_5^v=k[l,m,s,t,v]$, and the map $\varphi\colon T_5\longrightarrow T_5^v, $ that sends $u$ to $v(t-l)+m$ and $l,m,s,t$ to themselves. We have $\varphi(f_2)=(t-l)(\alpha t+\beta)$, where $$\alpha=l^2sv^2-l^2v^2-2lmsv+2lsv+m^2s+m^2-2msv-sv^2+2sv-s+v^2 -1,$$
and 
\begin{multline}\beta=l^3sv^2-l^3v^2-2l^2msv+2l^2mv+lm^2s-lm^2-2lmsv-lsv^2+2lsv\nonumber\\
-ls+lv^2+l+2m^2s-2mv+2sv-2s.\end{multline}
The map $\varphi$ induces a birational morphism $\psi\colon\mathbb{A}^5_v\longrightarrow\mathbb{A}^5,$ where $\mathbb{A}^5_v$ is the affine space with coordinate ring $T_5^v$. Moreover, $\psi$ is an isomorphism on the complements of the zero sets of $t-l$ in its domain and codomain. Set $$G=Z(\alpha t+\beta)\setminus Z(t-l)\subset\mathbb{A}^5_v,$$ then $\psi$ induces an isomorphism $G\cong F_2\setminus Z(t-l)$. In particular, $\psi$ induces an isomorphism from $G\setminus Z(\Delta)$ to $F_2'$. We want to show that $G\setminus Z(\alpha\Delta)$ equals $G\setminus Z(\Delta)$; to do this it suffices to show that $\psi(G \cap Z(\alpha))$ is contained in a union of singular fibers of~$\pi$. Note that we have $G\cap Z(\alpha)=G\cap Z(\alpha,\beta)$. Let $(l_0,m_0,s_0,t_0,v_0)$ be a point in $G\cap Z(\alpha,\beta)$, then, since $\alpha$ and $\beta$ do not depend on $t$, the point $(l_0,m_0,s_0,t,v_0)$ is contained in $Z(\alpha t+\beta)$ for all $t$. It follows that the fiber on $F_2$ in $\mathbb{A}^2(t,u)$ under~$\pi$ above the point $(l_0,m_0,s_0)\in\mathbb{A}^3$ contains the line $u=v_0(t-l_0)+m_0$, hence is singular. Moreover, this fiber contains the point $\psi((l_0,m_0,s_0,t_0,v_0))$. We conclude that $\psi(G\cap Z(\alpha))$ is contained in a union of singular fibers of $F_2$. It follows that $$\psi(G\setminus Z(\alpha\Delta))=\psi(G\setminus Z(\Delta))=F_2'.$$
Consider the ring $T_{4}=k[l,m,s,v]$, and let $K_4$ be its field of fractions. Consider the ring homomorphism $\rho\colon T_5^v\longrightarrow K_4$ that sends $t$ to~$\tfrac{-\beta}{\alpha}$, and $l,m,s,v$ to themselves. This induces a birational map $i\colon \mathbb{A}^4\longrightarrow Z(\alpha t+\beta)\subset\mathbb{A}^5_v$, where $\mathbb{A}^4$ is the affine space with coordinate ring $T_4$. The map $i$ induces an isomorphism from $\mathbb{A}^4\setminus Z(\alpha)$ to $Z(\alpha t +\beta)\setminus Z(\alpha)$; this isomorphism sends the zero set of $\Delta$ in $\mathbb{A}^4\setminus Z(\alpha)$ to the zero set of $\Delta$ in $Z(\alpha t +\beta)\setminus Z(\alpha)$, and the zero set of $t-l$ in $Z(\alpha t +\beta)\setminus Z(\alpha)$ corresponds to the zero set of $\alpha l+\beta$ in $\mathbb{A}^4\setminus Z(\alpha)$. Hence, we have an isomorphism $$\mathbb{A}^4\setminus Z(\alpha\Delta(\alpha l+\beta))\cong G\setminus Z(\alpha\Delta).$$ We conclude that we have an isomorphism $$\psi\circ i\colon\mathbb{A}^4\setminus Z(\alpha\Delta(\alpha l+\beta))\longrightarrow F_2'.$$
Recall that our aim is to show that $A_0\cap F_2$ is contained in $S_0$. Since we showed that all components of $F_2\setminus F_2'$ are contained in $S_0$, we have $A_0\cap F_2\subset S_0$ if and only if $A_0\cap F_2'\subset S_0$. Moreover, after setting $$A_1=i^{-1}(\psi^{-1}(A_0\cap F_2'))\mbox{ and }S_1=i^{-1}(\psi^{-1}(S_0\cap F_2')),$$ showing $A_0\subseteq S_0$ is equivalent to showing $A_1\subseteq S_1$.\\
For $i$ in $\{1,2,3,4\}$, the expression stating that $P$ is contained in $D_i$ is given by det$(H_i)=0$, where $H_i$ is the matrix denoted by $H$ in Lemma \ref{genpos1} associated to
\begin{align*}
&(R_2,R_3,R_4,R_5,R_6,R_1,R_7,R_8)\mbox{ for }i=1;\\
&(R_1,R_3,R_4,R_7,R_8,R_2,R_5,R_6)\mbox{ for }i=2;\\
&(R_1,R_2,R_4,R_5,R_7,R_3,R_6,R_8)\mbox{ for }i=3;\\
&(R_1,R_2,R_3,R_6,R_8,R_4,R_5,R_7)\mbox{ for }i=4,
\end{align*}
where we set $(\alpha_i,\beta_i,\gamma_i)=(x,y,z)$ for $i\in\{1,2\}$, and  
$(\alpha_i,\beta_i,\gamma_i)=(y,x,z)$ for $i$ in $\{3,\ldots,8\}$. For $i\in\{1,2,3,4\}$, let $B_i\subset F_2\subset\mathbb{A}^5$ be the locus of points corresponding to configurations of $R_1,\ldots,R_8$ such that $D_i$ contains $P$. Then $A_0\cap F_2=\bigcap_{i=1}^4 B_i$, so $A_0\cap F_2'=\bigcap_{i=1}^4(B_i\cap F_2')$, and hence $A_1=\bigcap_{i=1}^4i^{-1}(\psi^{-1}(B_i\cap F_2'))$.
Note that~$B_i$ is defined by $f_2=$det$(H_i)=0$. For $i\in\{1,2,3,4\}$, we compute the determinant of~$H_i$ and its factorization in $\mathbb{Z}[l,m,s,t,u]$ in \texttt{magma}. For all $i$, this factorization has a constant factor that is a power of 2, and there is exactly one irreducible factor~$h_i$ that does not define a component of $S_0$; it follows that $Z(f_2,h_i)\setminus S_0=B_i\setminus S_0.$ Note that for $i\in\{1,2,3,4\}$, the set $i^{-1}(\psi^{-1}(Z(f_2,h_i)\setminus Z(\alpha\Delta(t-l)))$ is defined in $\mathbb{A}^4\setminus Z(\alpha\Delta(\alpha l+\beta))$ by the numerator of $\rho(\varphi(h_i))$; we compute the factorization of this numerator in $\mathbb{Z}[l,m,s,v]$. Again, for all $i$, this factorization has as constant factor a power of~2, and contains exactly one irreducible factor that does not define a component of~$S_1$; we call this factor~$g_i$. It follows that for $i\in\{1,2,3,4\}$, the set $i^{-1}(\psi^{-1}(B_i\setminus S_0))$ is contained in $Z(g_i)$, so $A_1\setminus S_1$ is contained in $Z(g_1,g_2,g_3,g_4)$.  
Computing $g_1,g_2,g_3,g_4$ takes \texttt{magma} over an hour, and these polynomials are too big to write down here; you can find them in \cite{Magmaa}. Set 
\begin{multline}
\delta=(ls-l-ms-m+2s)^2(l-m)(l-s)(l+1)(m-1)(s+1)\cdot\nonumber\\
(l-1)(m+1)(s-1)v^2.
\end{multline} 
We check that all factors of $\delta\in\mathbb{Z}[l,m,s,v]$ define components of $S_1$ (the first factor corresponds to both $R_2,R_3,R_5$ and $R_2,R_4,R_6$ being collinear). We will show that~$\delta$ is contained in $I$. We use a Gr\"obner basis for~$I$ to check this. In \texttt{magma}, we define the ideal $I$ in the ring $T_4$ with $k=\mathbb{Q}$ with the ordering $s>v>m>l$. With the function $$\mbox{ \texttt{G,b:=GroebnerBasis(I:ReturnDenominators)}}$$ we compute the reduced Gr\"obner basis $G$ for~$I$; after using this function, \texttt{magma} uses $G$ as a generator set for $I$. We then use $G$ to check that $\delta$ is contained in $I$, again over~$\mathbb{Q}$. This finishes the proof for char $k=0$; We continue the proof for char~$k=p>0$ with $p\neq2,3$. \\
The element $\delta$ can be written as a linear combination of the elements in $G$ with coefficients in $T_4$. Let $C$ be the set of these coefficients (obtained by the function \texttt{Coordinates(I,f)}). In the proces of computing $G$, \texttt{magma} makes divisions by integers, which are stored in the set $b$. Let $\mathcal{P}$ be the set containing the prime divisors of all elements in $b$, and all prime divisors of the denominators of the coefficients of the elements in $G$, and all prime divisors of the denominators of the coefficients of the elements in $C$. Then for a prime $p\not\in\mathcal{P}$, the reductions modulo $p$ of the elements in $G$ are well defined. Moreover, since $\mathcal{P}$ contains all prime divisors of the elements in $b$, the reductions modulo $p$ of the elements in $G$ still form a Gr\"obner basis for the ideal $J$ generated by the reductions modulo $p$ of $g_1,g_2,g_3,g_4$. Finally, the reduction modulo $p$ of $\delta$ is contained in $J$, since the prime divisors of the denominators of the coefficients of the elements in $C$ are in $\mathcal{P}$. This finishes the proof for char $k=p>0$ with $p\neq 2,3$, $p\notin\mathcal{P}$.\\
For all finitely many $p\in\mathcal{P}\setminus\{2,3\}$, let $\overline{T_4}$ be the ring $\mathbb{F}_p[l,m,s,v]$, let $\overline{\delta}$ be the reduction of $\delta$ modulo~$p$, and for $i\in\{1,2,3,4\}$, let $\overline{g_i}$ be the reduction of $g_i$ modulo~$p$; then it is a quick check in $\texttt{magma}$ that $\overline{\delta}$ is contained in the ideal $(\overline{g_1},\overline{g_2},\overline{g_3},\overline{g_4})$ of~$\overline{T_4}$. 
We conclude that for char $k\neq2,3$, the set $A_1\setminus S_1$ is contained in the union of the varieties defined by the factors of $\delta$, so $A_1\setminus S_1$ is a subset of $S_1$. We conclude that~$A_1$ is contained in $S_1$. This finishes the proof for char~$k\neq2$. 

\vspace{5pt}

\underline{Assume char $k=2$.}\\
Since the points $R_1,R_5,R_6,P$ as defined in the previous case are not in general position over a field of characteristic 2, we redefine these points here. The proof then goes completely analogous to the previous case; see \cite{Magmab} for the code in \texttt{magma} where we verify everything over the field $k=\mathbb{F}_2$ of two elements. Set
\begin{align*}
&R_1=(1:0:1); &	R_6&=(0:1:1);\\
&R_5=(0:1:0); &	P&=(1:0:0).
\end{align*} These four points are in general position in $\mathbb{P}^2$. For the two generators of the linear system of quadrics through $R_1,\;R_5,\;R_6$ and $P$ we take $z^2+xz+yz$ and $xy$. \\
We now do all the steps as in the previous case, and everything works analogously. In fact, checking that all singular fibers of the analog of $\pi$ from the previous case are contained in the analog of $S_0$ can be done even more directly in \texttt{magma} than as described in the previous case. We obtain again an algebraic set $A_1\subset\mathbb{A}^4$, where $\mathbb{A}^4$ is the affine space over $\mathbb{F}_2$ with coordinates $l,m,s,v$, and $A_1$ is the algebraic set corresponding to the configurations where the ten curves $L_1,L_2,C_1,\ldots,C_4,D_1,\ldots,D_4$ all contain the point $P$. Again, we want to show that $A_1$ is contained in $S_1$, where $S_1\subset\mathbb{A}^4$ is the algebraic set defined by the polynomials that correspond to the eight points $R_1,\ldots,R_8$ not being in general position. Completely analogously to the case char $k\neq2$, from the conditions that~$P$ is contained in $D_1,D_2,D_3,D_4$, we now obtain four polynomials $g_1,g_2,g_3,g_4$ in $\mathbb{F}_2[l,m,s,v]$ (see \cite{Magmaa}).
Again, we have $A_1\setminus S_1\subset Z(g_1,g_2,g_3,g_4)$. Set $$\delta=(ls+ms+m+s)(lv+m+1)(l+m)(l+s)(m+s)(l+1)(m+1)m^3(s+1)lvs
.$$ It is a quick check with \texttt{magma} that $\delta$ is contained in $I$. Moreover, it is again a quick check that all factors of $\delta$ correspond to three points being collinear, and hence define a component of $S_1$. We conclude again that~$A_1$ is contained in $S_1$. \qed
\end{proofofkey2}

We can now prove Theorem \ref{thm2}.

\begin{proofoftheorem2}Recall that every set of exceptional curves without partners corresponds to a clique in $G$ with only edges of weights 1 and~2, so by Proposition \ref{max12bvt}, the number of exceptional curves that are concurrent in a point outside the ramification curve of $\varphi$ is at most twelve. This proves the case char~$k=3$. \\
Consider the eleven classes in $C$ given by 
\begin{align*}
&e_1=L-E_1-E_2;\\
&e_2=L-E_3-E_4;\\
&e_3=2L-E_1-E_3-E_5-E_6-E_7;\\
&e_4=2L-E_1-E_4-E_5-E_6-E_8;\\
&e_5=2L-E_2-E_3-E_5-E_7-E_8;\\
&e_6=2L-E_2-E_4-E_6-E_7-E_8;\\
&e_7=4L-2E_1-E_2-E_3-E_4-E_5-E_6-2E_7-2E_8;\\
&e_8=4L-E_1-2E_2-E_3-E_4-2E_5-2E_6-E_7-E_8;\\
&e_9=4L-E_1-E_2-2E_3-E_4-E_5-2E_6-E_7-2E_8;\\
&e_{10}=4L-E_1-E_2-E_3-2E_4-2E_5-E_6-2E_7-E_8;\\
&e_{11}=5L-2E_1-2E_2-2E_3-2E_4-2E_5-E_6-E_7-2E_8;
\end{align*}
It is straightforward to check that they form a clique with only edges of weights 1 and 2 in $G$. By Remark~\ref{configuratie}, we know that $e_1,\ldots,e_{10}$ correspond to the classes in Pic $X$ of the strict transforms of the curves $L_1,L_2,C_1,\ldots,C_4,D_1,\ldots,D_4$, defined as above Proposition \ref{key2} with respect to $P_i$ instead of $R_i$ for $i\in\{1,\ldots,8\}$.  \\
Let $K=\{c_1,\ldots,c_{11}\}$ be a clique of size eleven in $G$ with only edges of weights 1 and~2. By Proposition \ref{trans11}, after changing the indices if necessary, there is an element $g\in G$ such that $c_i=g(e_i)$ for $i\in\{1,\ldots,11\}$. Set $E_i'=g(E_i)$. Then, since the $E_i'$ are pairwise disjoint, by Lemma \ref{BD} we can blow down $E_1',\ldots,E_8'$ to points $Q_1,\ldots,Q_8$ in $\mathbb{P}^2$ that are in general position, such that $X$ is isomorphic to the blow-up of $\mathbb{P}^2$ at $Q_1,\ldots,Q_8$, and $E_i'$ is the class in Pic $X$ that corresponds to the exceptional curve above $Q_i$ for all~$i$. By the bijection in Remark \ref{configuratie}, the elements $c_1,\ldots,c_{10}$ are the classes that correspond to the strict transforms of $L_1,L_2,C_1,\ldots,C_4,D_1,\ldots,D_4$ defined as above Proposition \ref{key2} with respect to $Q_i$ instead of $R_i$ for $i\in\{1,\ldots,8\}$. If char $k\neq3$, it follows from Proposition \ref{key2} that the curves corresponding to $c_1,\ldots,c_{10}$ are not concurrent. We conclude that the number of concurrent exceptional curves in a point outside the ramification curve of $\varphi$ is less than eleven for char $k\neq3$.
\qed
\end{proofoftheorem2}

\section{Examples}\label{examples}

\subsection{On the ramification curve}

This section contains examples that show that the upper bounds in Theorem \ref{thm1} are sharp. Example \ref{char2} is a del Pezzo surface over a field of characteristic 2 with 16 concurrent exceptional curves, Example \ref{ex1} is a del Pezzo surface over any field of characteristic unequal to  $2,3,5,7,11,13,17,19$ with 10 concurrent exceptional curves, and Example \ref{char5711311719} contains examples of ten concurrent exceptional curves on del Pezzo surfaces in the remaining 7 characteristics.  

\begin{example}\label{char2}
Set $f=x^5+x^2+1\in \mathbb{F}_2[x]$, and let $F\cong\mathbb{F}_2[x]/(f)$ be the finite field of 32 elements defined by adjoining a root $\alpha$ of $f$ to $\mathbb{F}_2$. Define the following eight points in $\mathbb{P}^2_{F}$.
\begin{align*}
&Q_1=(0:1:1); & &Q_5=(1:1:1);\\
&Q_2=(0:1:\alpha^{19}); & &Q_6=(\alpha^{20}:\alpha^{20}:\alpha^{16});\\
&Q_3=(1:0:1); & &Q_7=(\alpha^{24}:\alpha^{25}:1);\\
&Q_4=(1:0:\alpha^5); & &Q_8=(\alpha^{30}:1:\alpha^5).
\end{align*}

With $\texttt{magma}$ we check that the determinants of the appropriate matrices in Lemma~\ref{genpos1} are all nonzero, so these eight points are in general position. Therefore, the blow-up of $\mathbb{P}^2$ in $\{Q_1,\ldots,Q_8\}$ is a del Pezzo surface $S$. 
We have the following four lines in~$\mathbb{P}^2$.
\begin{align*}
&\mbox{The line }L_1 \mbox{ through }Q_1\mbox{ and }Q_2,\mbox{ which is given by }x=0;\\
&\mbox{the line }L_2 \mbox{ through }Q_3\mbox{ and }Q_4,\mbox{ which is given by }y=0;\\
&\mbox{the line }L_3 \mbox{ through }Q_5\mbox{ and }Q_6,\mbox{ which is given by }x=y;\\
&\mbox{the line }L_4 \mbox{ through }Q_7\mbox{ and }Q_8,\mbox{ which is given by }y=\alpha x.
\end{align*}
Let $C_{i,j}$ be the unique cubic through $Q_1,\ldots,Q_{i-1},Q_{i+1},\ldots,Q_8$ that is singular in~$Q_j$. Set $(R_1,\ldots,R_8)=(Q_1,Q_3,Q_4,Q_5,Q_6,Q_7,Q_8,Q_2)$, and let $L$ be the corresponding matrix from Lemma \ref{genpos1}. Then the equation defining $C_{1,2}$ is the determinant of $L'$, where $L'$ is equal to $L$ after replacing the first row by Mon$_3$. Similarly, we compute the defining equations of $C_{3,4},\;C_{5,6},\;C_{7,8}$ and $C_{8,7}$, and find the following. 
\begin{multline*}C_{1,2}\colon 
x^3+\alpha^{24}x^2y+\alpha^{28}x^2z+\alpha^{30}xy^2+\alpha^9xyz+\alpha^{26}xz^2+\alpha^{13}y^3+\alpha^6yz^2=0
\end{multline*}
\begin{multline*}
C_{3,4}\colon
x^3+\alpha^{12}x^2y+\alpha^4xy^2+\alpha^{11}xyz+\alpha^{21}xz^2+y^3+\alpha^{23}y^2z+\alpha^{12}yz^2=0
\end{multline*}
\begin{multline*}
C_{5,6}\colon
x^3+ \alpha^4 x^2y+\alpha^{28} x^2z+\alpha^{25}xy^2+\alpha^{20}xyz+\alpha^{26}xz^2+\alpha^{17}y^3+\alpha^9y^2z+\alpha^{29}yz^2=0
\end{multline*}
\begin{multline*}
C_{7,8}\colon
x^3+ \alpha x^2y+\alpha^{28} x^2z+\alpha^{17}xy^2+\alpha^{10}xyz+\alpha^{26}xz^2+\alpha^{16}y^3+\alpha^8y^2z+\alpha^{28}yz^2=0
\end{multline*}
\begin{multline*}
C_{8,7}\colon 
x^3+ \alpha^{26} x^2y+\alpha^{28} x^2z+\alpha^{19}xy^2+\alpha^{10}xyz+\alpha^{26}xz^2+\alpha^{16}y^3+\alpha^8y^2z+\alpha^{28}yz^2=0
\end{multline*}

Let $e_1,\ldots,e_8$ be the strict transforms of the eight curves $$L_1,\ldots,L_4,C_{1,2},C_{3,4},C_{5,6},C_{7,8},$$ and let $c_8$ be the strict transform of $C_{8,7}$. Since these nine curves all contain the point $(0:0:1)$, the exceptional curves $e_1,\ldots,e_8,c_8$ are concurrent in a point $P$ on~$S$. Let~$\psi$ be the morphism associated to the linear system $|-2K_S|$.   Since $e_8\cdot c_8=3$, the point $P$ lies on the ramification curve of $\psi$ by Remark \ref{rem}. Therefore, by the same remark, for $i\in\{1,\ldots,7\}$, the partners of $e_1,\ldots,e_7$ contain $P$, too. We conclude that there are sixteen exceptional curves on~$S$ that are concurrent in $P$. 
\end{example}

\begin{example}\label{ex1}
Let $k$ be a field of characteristic unequal to $2,3,5,7,11,13,17,19$. Define the following eight points in $\mathbb{P}^2_{k}$.
\begin{align*}
&Q_1=(0:1:1); & &Q_5=(1:1:1);\\
&Q_2=(0:5:3); & &Q_6=(4:4:5);\\
&Q_3=(1:0:1); & &Q_7=(-2:2:1);\\
&Q_4=(-1:0:1); & &Q_8=(2:-2:1).
\end{align*}
With \texttt{magma} we compute the determinants of the matrices in Lemma \ref{genpos1} that determine whether three of the points are on a line, or six of the points are on a conic, or seven of them are on a cubic that is singular at one of them. These determinants are nonzero for char~$k\neq2,3,5,7,11,13,17,19$, so the points are in general position.  
Therefore, the blow-up of $\mathbb{P}^2_k$ in $\{Q_1,\ldots,Q_8\}$ is a del Pezzo surface~$S$. 
We define the lines $L_1,L_2,L_3$ as in Example \ref{char2}. We define $L_4$ to be the line containing $Q_7$ and $Q_8$, which is given by $x=-y$.\\
Let $C_{7,8}$ be the unique cubic through $Q_1,\ldots,Q_6,Q_8$ that is singular in $Q_8$, and $C_{8,7}$ the unique cubic through $Q_1,\ldots,Q_7$ that is singular in $Q_7$. 
As in Example \ref{char2} we compute the defining equations for $C_{7,8}$ and $C_{8,7}$, and we find
\begin{align*}&C_{7,8}\colon x^3-\tfrac{3}{4}x^2y-\tfrac{31}{12}xy^2+\tfrac{10}{3}xyz-xz^2-y^3+\tfrac{8}{3}y^2z-\tfrac{5}{3}yz^2=0,\\
&C_{8,7}\colon x^3+\tfrac{13}{4}x^2y+\tfrac{43}{4}xy^2-14xyz-xz^2+15y^3-40y^2z+25yz^2=0.
\end{align*}
On $S$, we define the four exceptional curves $e_1,\ldots,e_4$ to be the strict transforms of $L_1,\ldots,L_4$, and $e_5,c_5$ the strict transforms of $C_{7,8}$ and $C_{8,7}$, respectively. 
Since $L_1,\ldots,L_4,C_{7,8},C_{8,7}$ all contain the point $(0:0:1)$, the six exceptional curves $e_1,\ldots,e_5,c_5$ are concurrent in a point $P$ in $S$. Let $\psi$ be the morphism associated to the linear system $|-2K_S|$. By Remark~\ref{rem}, since $e_5\cdot c_5=3$, the point~$P$ lies on the ramification curve of $\psi$, and for $i\in\{1,\ldots,4\}$, the partners of $e_1,\ldots,e_4$ contain~$P$, too. We conclude that there are ten exceptional curves on $S$ that are concurrent in~$P$. 
\end{example}

\begin{example}\label{char5711311719}for $p\in\{3,5,7,11,13,17,19\}$, we construct a del Pezzo surface over a field of characteristic $p$ with ten exceptional curves that are concurrent in a completely analogous way to the one in Example~\ref{ex1}. \\
Let $p$ be a prime, and $\mathbb{F}_p$ be the finite field of $p$ elements. Let $f_p\in\mathbb{F}_p[x]$ be an irreducible polynomial. Let $\alpha$ be a root of $f_p$, and $\mathbb{F}\cong\mathbb{F}_p[x]/f_p$ the field extension of~$\mathbb{F}_p$ obtained by adjoining $\alpha$ to $\mathbb{F}_p$. 
For $a,b,c,m,u,v\in\mathbb{F}$, define the following eight points in $\mathbb{P}^2_{\mathbb{F}}$. 
\begin{align*}
&Q_1=(0:1:1); & &Q_5=(1:1:1);\\
&Q_2=(0:1:a); & &Q_6=(1:1:c);\\
&Q_3=(1:0:1); & &Q_7=(m:1:u);\\
&Q_4=(1:0:b); & &Q_8=(m:1:v).
\end{align*}
Let $x,y,z$ be the coordinates of $\mathbb{P}^2_{\mathbb{F}}$. We define again the lines $L_1,L_2,L_3$ as in Example~\ref{char2}, and the line $L_4$ by $x=my$. Note that $L_1,\ldots,L_4$ all contain the point~$(0:0:1)$. Let $C_{7,8}$ be the unique cubic through $Q_1,\ldots,Q_6,Q_8$ that is singular in $Q_8$, and $C_{8,7}$ the unique cubic through $Q_1,\ldots,Q_7$ that is singular in~$Q_7$. 
For all fixed $(p,f_p,a,b,c,m,u,v)$ that we describe below, we check as we did in Example \ref{ex1} that the eight points are in general position, and compute the defining equations for $C_{7,8}$ and $C_{8,7}$. In all cases, the point $(0:0:1)$ is also contained in $C_{7,8}$ and $C_{8,7}$, and as in Example \ref{ex1} this implies that there are 10 exceptional curves on the del Pezzo surface obtained by blowing up $\mathbb{P}^2_{\mathbb{F}}$ in $Q_1,\ldots,Q_8$, that are concurrent in a point on the ramification curve.

\vspace{5pt}

$\bullet$ For $p=3$ we take $$f_p=x^3 + 2x + 1,\;\;(a,b,c,m,u,v)=(\alpha,\alpha^{20},\alpha^{15},\alpha^{8},\alpha^2,\alpha^{12}).$$ 
$\bullet$ For $p=5$ we take $$f_p=x^2+4x+2,\;\;(a,b,c,m,u,v)=(\alpha^{19},\alpha^{11},\alpha^{10},\alpha^{21},\alpha^3,\alpha^{14}).$$
$\bullet$ For $p=7$ we take $$f_p=x^2+6x+3,\;\;(a,b,c,m,u,v)=(3,\alpha^{45},\alpha^{35},\alpha^{4},\alpha^{46},\alpha^9).$$
 $\bullet$ For $p=11$ we take $$f_p=x^2+7x+2,\;\;(a,b,c,m,u,v)=(\alpha^{106},\alpha^{94},4,\alpha^{62},\alpha^{111},\alpha^{6}).$$
$\bullet$ For $p=13$ we take $$f_p=x^2+12x+2,\;\;(a,b,c,m,u,v)=(\alpha^{161},\alpha^{156},\alpha^{83},\alpha^{94},\alpha^{132},\alpha^{146}).$$
$\bullet$ For $p=17$ we take $$f_p=x^2+16x+3,\;\;(a,b,c,m,u,v)=(\alpha^{74},\alpha^{166},\alpha^{64},\alpha^{24},\alpha^{178},\alpha^{250}).$$
$\bullet$ For $p=19$, we take $\mathbb{F}=\mathbb{F}_{19}$, and 
$(a,b,c,m,u,v)=(2,2,14,8,7,12)$.

\vspace{5pt}

All these examples are generated in \texttt{magma} by generating random values for the elements $a,b,c,m,u,v$ in each case, until the points defined by the values are in general position. 
\end{example}

\subsection{Outside the ramification curve}
In this section we give examples that show that the upper bound in Theorem \ref{thm2} is sharp. Example \ref{char3} gives a del Pezzo surface of degree one over a field of characteristic 3 with twelve exceptional curves that are concurrent in a point outside the ramification curve. In Example \ref{ex2} we give a del Pezzo surface over a field of characteristic unequal to 5 that contains ten exceptional curves that are concurrent in a point outside the ramification curve. This surface is isomorphic to the one in Example 4.1 in \cite{SL14} if the characteristic of $k$ is unequal to $2$ and $3$. We do not give an example in characteristic 5, since we have not found one; it might very well be that the maximum in this case is less than ten.

\begin{example}\label{char3}Let $f=x^3+2x+1$ be a polynomial in $\mathbb{F}_3[x]$. Let $\alpha$ be a root of $f$, and let $\mathbb{F}\cong\mathbb{F}_3[x]/f$ be the field of 27 elements obtained by adjoining $\alpha$ to $\mathbb{F}_3$. Let $\mathbb{P}^2_{\mathbb{F}}$ be the projective plane over $\mathbb{F}$, and define the following eight points in this plane.
\begin{align*}
&Q_1=(1:0:1); & &Q_5=(0:1:1);\\
&Q_2=(\alpha^{20}:0:\alpha^{18}); & &Q_6=(0:2:1);\\
&Q_3=(\alpha^{6}:\alpha^{23}:\alpha^{2}); & &Q_7=(\alpha^{9}:\alpha^{23}:2);\\
&Q_4=(\alpha^{15}:\alpha^{19}:\alpha^{18}); & &Q_8=(\alpha^{24}:\alpha^{7}:\alpha^{5}).
\end{align*}

With \texttt{magma} we check that no three of these points are on a line, no six of them are on a conic, and no seven of them are on a cubic that is singular at one of them, by checking that the appropriate determinants of the matrices in Lemma \ref{genpos1} are nonzero. Therefore, the blow-up of $\mathbb{P}^2_{\mathbb{F}}$ in these eight points is a del Pezzo surface $S$ of degree one. \\
Let $L_1$ be the line containing $Q_1$ and $Q_2$, which is given by $y=0$. Let $L_2$ be the line containing $Q_3$ and $Q_4$, which is given by $\alpha^{23}y=x+z$. 
For five points $Q_{i_1},\ldots,Q_{i_5}$ we find the equation of the conic containing these points by computing the determinant of the matrix $N$ in Lemma \ref{genpos1}, with $(R_2,\ldots,R_6)=(Q_{i_1},\ldots,Q_{i_5})$, and where the first row is replaced by the list Mon$_2$. We obtain the following conics in $\mathbb{P}^2_{\mathbb{F}}$.  

\vspace{5pt}

$C_1\colon x^2 + \alpha^7xy + y^2 + 2z^2=0,\mbox{ containing }Q_1,Q_3,Q_5,Q_6,Q_7.$\\
$C_2\colon x^2 + \alpha^{16}xy + y^2 + 2z^2=0,\mbox{ containing }Q_1,Q_4,Q_5,Q_6,Q_8.$\\
$C_3\colon x^2 + \alpha^{25}xz + \alpha^{16}y^2 + \alpha^{11}yz + \alpha^{15}z^2=0,\mbox{ containing }Q_2,Q_3,Q_5,Q_7,Q_8.$  \\ 
$C_4\colon x^2 + \alpha^9xy + \alpha^{25}xz + \alpha^{20}y^2 + \alpha^6yz + \alpha^{15}z^2=0,
 \mbox{ containing }Q_2,Q_4,Q_6,Q_7,Q_8.$
 
\vspace{5pt}

Similarly, we compute defining equations for the quartics $D_1,D_2,D_3,D_4$ containing all the eight points with singularities in $Q_1,Q_7,Q_8$, and $Q_2,Q_5,Q_6$, and $Q_3,Q_6,Q_8$, and $Q_4,Q_5,Q_7$, respectively. We find
\begin{multline*}   
D_1\colon \alpha^4x^4 + \alpha^{11}x^3y + \alpha^{12}x^3z + \alpha^{24}x^2y^2 + \alpha^{10}x^2yz +
    \alpha^{16}x^2z^2 + \alpha^{16}xy^3 + \alpha^{21}xy^2z \\
    + \alpha^{17}xyz^2 + 
    \alpha^{25}xz^3 + \alpha^6y^4 + \alpha^{12}y^3z + \alpha^{25}yz^3 + \alpha^{19}z^4=0,
\end{multline*}    
\begin{multline*}   
D_2\colon \alpha^{14}x^4 + x^3y + \alpha^{16}x^3z + \alpha^4x^2y^2 + \alpha^4x^2yz + 
    \alpha^{21}x^2z^2 + \alpha^{25}xy^3 + \alpha^{16}xy^2z  \\
  +   \alpha^{12}xyz^2 + 
    \alpha^3xz^3 + \alpha^5y^4 + \alpha^5y^2z^2 + \alpha^5z^4=0,
\end{multline*}    
\begin{multline*}   
D_3\colon \alpha^{21}x^4 + \alpha^4x^3y + \alpha^{20}x^3z + \alpha^9x^2y^2 + \alpha^{19}x^2yz + 
    \alpha^3x^2z^2 + \alpha^{21}xy^3
     + \alpha^{11}xy^2z \\
     + \alpha^2xyz^2 + 
    \alpha^7xz^3 + \alpha^2y^4 + \alpha^{17}y^3z + \alpha y^2z^2 + \alpha^4yz^3 + 
    \alpha^{23}z^4=0,
\end{multline*} 
\begin{multline*}   
 D_4\colon \alpha^{19}x^4 + \alpha^{22}x^3y + \alpha^{18}x^3z + \alpha^{20}x^2y^2 + \alpha^{21}x^2yz 
    + \alpha x^2z^2 + \alpha^2xy^3 + \alpha^{20}xy^2z \\ 
    + \alpha^{10}xyz^2     + \alpha^5xz^3 + \alpha^{23}y^4 + \alpha^{20}y^3z + \alpha^3y^2z^2 + \alpha^7yz^3 +
    \alpha^{21}z^4=0.
\end{multline*}
Finally, in a similar way we compute the defining equations of the quintics $G_1$ and $G_2$, which contain all eight points and are singular in $Q_1,Q_2,Q_3,Q_4,Q_5,Q_8$, and $Q_1,Q_2,Q_3,Q_4,Q_6,Q_7$, respectively. We obtain 
\begin{multline*}
G_1\colon \alpha x^5 + \alpha^8x^4y + 2x^4z + \alpha^{21}x^3y^2 + \alpha^{20}x^3yz + 
    \alpha^{23}x^3z^2 + \alpha^5x^2y^3 + \alpha^{25}x^2y^2z \\
    + \alpha^{22}x^2yz^2 + \alpha^7x^2z^3 + \alpha^{25}xy^4 + \alpha^{12}xy^3z + 2xy^2z^2 + 
    \alpha^{25}xyz^3 + \alpha^2xz^4 \\
    + \alpha^{21}y^5 + \alpha^6y^4z 
    + \alpha^8y^3z^2
    + \alpha y^2z^3 + \alpha^5z^5=0,
\end{multline*}
\begin{multline*}
G_2\colon \alpha^4x^5 + \alpha^{11}x^4y + \alpha^{16}x^4z + \alpha^7x^3y^2 + \alpha^{16}x^3yz + 
    x^3z^2 + \alpha x^2y^3 + \alpha^{25}x^2y^2z \\
    + \alpha^2x^2yz^2 + 
    \alpha^{10}x^2z^3 + \alpha^{17}xy^3z +
     \alpha^{15}xy^2z^2 + \alpha^8xyz^3 
     + 
    \alpha^5xz^4 
    + \alpha^{14}y^5 \\
    + \alpha^{16}y^4z + \alpha^{11}y^3z^2 + 
    \alpha^{10}y^2z^3 + \alpha^{25}yz^4 + \alpha^8z^5=0.
\end{multline*}
Now consider the point $P=(2:0:1)$ in $\mathbb{P}^2_{\mathbb{F}}$. It is an easy check that $P$ is contained in all twelve curves $L_1,L_2,C_1,\ldots,C_4,D_1,\ldots,D_4,G_1,G_2$. Therefore, the twelve exceptional curves on $S$ that are the strict transforms of these twelve curves in $\mathbb{P}^2_{\mathbb{F}}$ are concurrent in a point $Q$ on $S$. Let $\psi$ be the morphism associated to the linear system $|-2K_S|$. Since none of the twelve exceptional curves intersect each other with multiplicity 3, the point $Q$ is outside the ramification curve of $\psi$.
\end{example}

\begin{example}\label{ex2}
Let $k$ be a field of characteristic unequal to 5. For $\beta$ an element in $k^*$, let $S$ be the del Pezzo surface of degree one in $\mathbb{P}(2,3,1,1)$ with coordinates $x,y,z,w$ over $k$ given by $$y^2+(\beta+1)xyw+\beta yw^3= x^3 + \beta x^2w^2 -z^5w.$$
For char $k\neq2,3$, this surface is isomorphic to the surface in \cite{SL14}, Example 4.1. The blow-up of $S$ in the point $(1:1:0:0)$ has the structure of an elliptic surface over $\mathbb{P}^1$ with coordinates $z,w$. The fiber above $z=0$ contains a point of order 5, which is given by $Q=(0:0:0:1)$; in fact, the cubic curve $$E:y^2+(\beta+1)xy+\beta y= x^3 + \beta x^2$$ is the universal elliptic curve over the modular curve $Y_1(5)=\mbox{Spec}\left(k[\beta,1/\Delta(E)]\right)$ with $\Delta(E)=-\beta^5(\beta^2+11\beta-1)$ that parametrizes elliptic curves over extensions of~$k$ with a point of order 5 (\cite{Edix11}, Proposition 8.2.8).  \\
Choose $\beta$ such that $S$ is smooth in all characteristics; for example, we can set $\beta=2$ in characteristic 11, and $\beta=1$ in all other characteristics. Let $\rho,\sigma$ be elements of a field extension of $k$ such that $\rho^2=\rho+1$, and $(\beta+\rho^5)\sigma^5=1$. Consider the curve $C_{\rho,\sigma}$ in $\mathbb{P}(2,3,1,1)$ defined by 
\begin{align*} 
&x=\sigma^2z^2w^4 + \rho\sigma zw^5,\\
&y=-\sigma^3z^3w^3+(\rho+1)\sigma^2z^2w^4.
\end{align*}
Then $C_{\rho,\sigma}$ is an exceptional curve in~$S$, defined over $k(\rho,\sigma)$. It is easy to see that~$Q$ is contained in $C_{\rho,\sigma}$. There are ten pairs $(\rho,\sigma)$, so we conclude that there are ten exceptional curves through $Q$ over a field extension of~$k$. Finally, let $\varphi$ be the morphism associated to $|-2K_S|$. Since the points on the ramification curve of $\varphi$ are exactly the points on $S$ that are 2-torsion on their fiber, we conclude that $Q$ is outside the ramification curve. 
\end{example}

\begin{remark}
In the previous example, the point $Q$ is torsion on its fiber of the elliptic surface associated to $S$ (obtained by blowing up the base point of the anticanonical linear system $|-K_S|$, which is $(1:1:0:0)$), and it is contained in a high number of exceptional curves on $S$. A natural question is whether a point contained in `many' exceptional curves is always torsion on its fiber (where `many' would need to be specified). A positive answer to this question, where we take `many' to be at least 9, seems intuitively true by the following argument, which was pointed out to us by several people. Let $X$ be a del Pezzo surface of degree 1 over a field $k$, and let~$P$ be a point on $X$ that is contained in at least 9 exceptional curves, say $L_1,\ldots,L_n$. These curves correspond to sections of the elliptic surface $\mathcal{E}$ associated to $X$ \cite[Lemma 10.9]{Sh90}, which in turn correspond to elements in the Mordell--Weil group of $\mathcal{E}$ (which is the Mordell--Weil group of the generic fiber seen as elliptic curve of the function field $k(t)$). This Mordell--Weil group has rank at most 8 over $k$ \cite[Theorem 10.4]{Sh90}, so in this group there must be a relation $a_1L_1+\cdots+a_nL_n=0$, where $a_1,\ldots,a_n\in \mathbb{Z}$. Since all $n$ exceptional curves contain the point $P$, on the fiber of $P$ this specializes to $(a_1+\cdots+a_n)P=0$. If one reasons too quickly, it seems that this proves that $P$ is torsion of order dividing $a_1+\cdots +a_n$ on its fiber. However, it might be the case that $a_1+\cdots +a_n=0$, so this does not prove a positive answer to our earlier question. With help of the results in \cite{WvL}, we can show that for $n\geq9$, if $L_1,\ldots,L_n$ is the maximal set of lines going through $P$, then there is always a relation between $L_1,\ldots,L_n$ in the Mordell--Weil group of $\mathcal{E}$ that specializes to a non-trivial relation on the fiber of $P$, thus implying that $P$ is torsion. See \cite[Chapter 5]{W21}. 
\end{remark}

\bibliographystyle{alpha}
\bibliography{DP1}

\vspace{11pt}

\addresseshere

\end{document}